\documentclass[reqno,12pt]{amsart}

\usepackage{amsmath}
\usepackage{amssymb}
\usepackage{graphicx}
\usepackage{xcolor}

\usepackage[shortlabels]{enumitem}

\DeclareFontShape{OT1}{cmr}{m}{scit}{<-> cmcsc10}{}

\usepackage[colorlinks=true]{hyperref}
\hypersetup{urlcolor=blue, citecolor=red}

\usepackage[capitalize,nameinlink,noabbrev]{cleveref}
\crefformat{equation}{#2(#1)#3}
\crefformat{enumi}{#2(#1)#3}

\newtheorem{theorem}{Theorem}[section]
\newtheorem{definition}{Definition}[section]
\newtheorem{proposition}{Proposition}[section]

\newtheorem{lemma}{Lemma}[section]

\newtheorem{corollary}{Corollary}[section]

\numberwithin{equation}{section}

\newcommand{\NN}{{\mathbb{N}}}

\newcommand{\RR}{{\mathbb{R}}}

\newcommand{\bfu}{\mathbf{u}}
\newcommand{\bfv}{\mathbf{v}}
\newcommand{\bfw}{\mathbf{w}}

\newcommand{\bff}{\mathbf{f}}

\newcommand{\bfx}{\mathbf{x}}

\newcommand{\bfD}{\mathbf{D}}

\newcommand{\bfF}{\mathbf{F}}

\newcommand{\bfzero}{\mathbf{0}}

\newcommand{\bfnabla}{\boldsymbol{\nabla}}

\newcommand{\calB}{{\mathcal B}}
\newcommand{\calC}{{\mathcal C}}

\newcommand{\calE}{{\mathcal E}}
\newcommand{\calF}{{\mathcal F}}

\newcommand{\calP}{{\mathcal P}}

\newcommand{\calT}{{\mathcal T}}

\newcommand{\radius}{{r}}
\newcommand{\Ex}{\calE}

\newcommand{\rmd}{\mathrm{d}}
\newcommand{\rmw}{\mathrm{w}}
\newcommand{\rmv}{\mathrm{v}}

\newcommand{\rmb}{\mathrm{b}}

\newcommand{\rmc}{\mathrm{c}}

\newcommand{\cyl}{{\mathrm{cyl}}}

\newcommand{\ens}{{\mathrm{z}}}
\newcommand{\fpsss}{{\mathrm{fpsss}}}
\newcommand{\fp}{{\mathrm{fp}}}

\newcommand{\average}[1]{\langle{#1}\rangle}
\newcommand{\dual}[1]{\langle{#1}\rangle}

\newcommand{\Lim}{\operatorname*{\textsf{Lim}}_{T\rightarrow \infty}}

\newcommand{\wconv}{\rightharpoonup}

\title[Optimal minimax bounds for ensemble averages of 3D NSE]{Optimal minimax formula for bounds on ensemble averages of statistically stationary three-dimensional Navier-Stokes flows}

\author[A. Bronzi]{Anne Bronzi}
\author[C. F. Mondaini]{Cec\'{\i}lia F. Mondaini}
\author[R. M. S. Rosa]{Ricardo M. S. Rosa}

\address[A. Bronzi]{Instituto de Matem\'atica, Estat\'istica e Computa\c c\~ao Cient\'ifica, Universidade Estadual de Campinas (UNICAMP), Brazil}
\address[C. F. Mondaini]{Department of Mathematics, Drexel University, USA}
\address[R. M. S. Rosa]{Instituto de Matem\'atica, Universidade Federal do Rio de Ja\-nei\-ro, Brazil}

\email[A. Bronzi]{acbronzi@unicamp.br}
\email[C. F. Mondaini]{cf823@drexel.edu}
\email[R. Rosa]{rrosa@im.ufrj.br}

\date{\today}


\begin{document}

\begin{abstract}
We establish an optimal upper bound formula for ensemble averages of flow quantities associated with the three-dimensional incompressible Navier-Stokes equations. The formula takes the form of a minimax problem, extending the framework developed by Tobasco, Goluskin, and Doering (2018) for finite-dimensional systems and by Rosa and Temam (2022) for the two-dimensional Navier-Stokes equations. The lack of global well-posedness for the 3D case presents a significant challenge, which we overcome by working within the space of Foias-Prodi stationary statistical solutions. The minimax formula is derived by exploiting suitable compactness and continuity properties of specific subspaces of probability measures under the weak topology of the phase space. A distinguishing feature of our result is the characterization of the maximizing measures: unlike the previous cases, the optimal bounds in 3D are achieved on extreme points that are specific convex combinations of at most two Dirac delta measures, instead of exactly one, a structure that naturally appears from the constraints given by the mean energy dissipation inequalities in the characterization of the Foias-Prodi stationary statistical solutions.
\end{abstract}

\maketitle

\tableofcontents

\section{Introduction}

In the conventional statistical theory of turbulence, one is interested in ensemble averages of relevant quantities of the flow, such as kinetic energy, enstrophy, energy dissipation rate, and so on. Estimating such mean quantities is of fundamental importance in both theory and applications of fluid flows. Here, we prove an optimal upper bound formula, in the form of a minimax problem, for bounds on ensemble quantities associated with the three-dimensional Navier-Stokes equations (NSE). This further extends, to the three-dimensional case, the optimal minimax formula first obtained by Tobasco, Goluskin, and Doering (2018) \cite{TobGolDoe2018} for (finite-dimensional) ordinary differential equations, and later extended by Rosa and Temam (2022) \cite{RosaTemam2022} to the two-dimensional NSE. The extension to the three-dimensional case was a major outstanding challenge, due to the lack of regularity.

The ensemble average $\average{\phi}$ of a quantity $\phi=\phi(\bfu)$ of the velocity field $\bfu=\bfu(t, \bfx)$ is informally defined as an average of a number of sample flows $\bfu_j=\bfu_j(t, \bfx)$, $j=1, \ldots, n$,
\[
    \average{\phi} = \frac{1}{n} \sum_{j=1}^n \phi(\bfu_j(t, \cdot)).
\]
This average may depend on $t$, as in the study of decaying turbulence, or be independent of $t$, as in the case of statistically stationary flows, which is the case of interest to us.

In a more formal way, this quantity is regarded as an average, or expectation, with respect to a probability distribution on a sample space, or, in the rigorous framework developed initially by Foias in the early 1970's \cite{Foias72,Foias73}, as a Borel probability measure $\mu$ on the phase space $H$ of the system (see \cref{secsetting} for the mathematical framework),
\[
    \average{\phi}_\mu = \int_H \phi(\bfu) \;\rmd\mu(\bfu).
\]
In the statistically stationary case, the probability distribution would be a time-invariant distribution. However, due to the lack of a known well-posedness result for the 3D NSE (see e.g. \cite{BuckmasterVicol2019,HouWangYang2025} and references therein, for recent results and discussions on this matter), this statistical time invariance is made rigorous with the notion of a \emph{stationary statistical solution,} in the Foias-Prodi framework \cite{Foias73}.


In many situations, one looks for upper (or lower) bounds for such quantities, in the form of $\average{\phi}_\mu \leq C,$ for all relevant $\mu,$ which can also be written as
\[
    \sup_\mu \average{\phi}_\mu \leq C.
\]

A ``sharp'' upper bound is obtained when the bound $C$ is close to the supremum in some suitable sense,
especially when regarded with respect to some relevant parameters, such as viscosity, Reynolds number, Grashof number, and so on. The holy grail, though, from a theoretical perspective, is an \emph{optimal estimate}, obtained when we have precisely an equality
\[
    \sup_\mu \average{\phi}_\mu = C.
\]

With that in mind, our aim in this work is to prove the optimal minimax formula
\begin{equation}
    \label{optminimaxintro}
    \max_\mu \average{\phi}_\mu = \inf_{\Psi} \max_{\bfu_1, \bfu_2, \theta_1, \theta_2} \;\sum_{i=1}^2\theta_i\{\phi(\bfu_i) + \dual{\bfF(\bfu_i), \bfD\Psi(\bfu_i)}_{V', V} \}.
\end{equation}
In this identity, $\Psi : H \rightarrow \mathbb{R}$ is a test functional on the phase space $H$ of square-integrable divergence-free vector fields; $\bfu_1,\bfu_2$ are divergence-free vector fields in a region of the phase space; $\theta_1 \in [0, 1]$ and $\theta_2 = 1 - \theta_1$ are weights for a convex combination; $\bfF=\bfF(\bfu)$ is the right hand side of the evolution equation (the 3D NSE, in our case); and $\dual{\cdot, \cdot}_{V', V}$ is the duality product for the space $V$ of divergence-free vector fields in $H_0^1(\Omega)^3$, with $V \subset H = H' \subset V'$ (see \cref{secsetting} for the mathematical framework). 

The right hand side of the optimal formula \eqref{optminimaxintro} involves a maximum for points over (a region of) the phase space, so it does not depend on solving the Navier-Stokes equations or the corresponding transport equation. The asymptotic statistical regime is captured by the test functional, which attempts to act, in the minimax program, as a form of Lyapunov function \cite{TobGolDoe2018,FGHC2016}.

The corresponding formula obtained in \cite{RosaTemam2022} for the two-dimensional Navier-Stokes equations (and similarly in \cite{TobGolDoe2018} for systems of ordinary differential equations) involves only one point in phase space and reads
\begin{equation}
    \label{optminimaxintroprevious}
    \max_\mu \average{\phi}_\mu = \inf_{\Psi} \max_{\bfu} \{\phi(\bfu) + \dual{\bfF(\bfu), \bfD\Psi(\bfu)}_{V', V} \}.
\end{equation}

The minimax problem in the right hand side of \eqref{optminimaxintroprevious} can be written as a linear convex optimization problem. This is an important property from a practical point of view. Upon approximating the test functionals by polynomials, the linear convex optimization problem becomes a semidefinite programming problem, which can be solved numerically with the help of suitable optimization softwares, as done for instance in \cite{CheGouHuaPa2014,FantuzziWynn2015,FGHC2016,TobGolDoe2018,Goluskin2018,GoluskinFantuzzi2019}, providing rigorous and sharp bounds for practical problems. Similarly, the minimax problem in the right hand side of \eqref{optminimaxintro} can also be written as a linear convex optimization problem, albeit with a slightly more involved structure, due to the presence of the convex combination of two points in the phase space (see \cref{secminimaxaslinearoptimization} for details).

In order to be more precise about \eqref{optminimaxintro}, we need to specify the allowed measures $\mu$, the allowed test functionals $\Psi$, and the allowed velocity fields $\bfu_1, \bfu_2$, in the maximum and infimum above. The choices are important. The measures $\mu$, in particular, need to represent the statistically stationary flow regimes; the test functionals are used to define these statistically stationary measures, in a weak sense; and the set for the velocity fields is related to the regularity of these statistically stationary measures.

We start with a bounded and closed subset $B$ of the phase space $H$ and assume it is weakly compact. We denote by $B_\rmw$ the set $B$ endowed with the weak topology of $H$. Next we consider the space $\calP(B_\rmw)$ of Borel probability measures on $H_\rmw$ which are carried by $B_\rmw$ and endowed with the topology of weak-star convergence of measures for weakly continuous functions on $B.$ 
This is a compact probability space. The measures associated with statistically stationary flows are taken to be \emph{Foias-Prodi stationary statistical solutions} \cite{FMRT2001,FRT2013,FRT2019}. They are measures $\mu\in \calP(H)$ with some regularity properties and associated with the Navier-Stokes equations via a Liouville-type transport equation, i.e. they have \emph{finite mean enstrophy},
\begin{equation}
    \label{finitemeanenstrophyintro}
    \int_H \|\bfnabla \bfu\|_{L^2}^2 \;\rmd\mu(\bfu) < \infty,
\end{equation}
and they satisfy the \emph{mean energy dissipation inequalities}
\begin{equation}
    \label{meandissipationineqsintro}
    \int_H \psi'(\|\bfu\|_{L^2}^2) \left\{\nu\|\bfnabla\bfu\|_{L^2}^2 - \dual{\bff, \bfu}_{V', V} \right\}\;\rmd\mu(\bfu) \leq 0,
\end{equation}
valid for all nonnegative, non-decreasing continuously-differentiable functions $\psi$ on $[0, \infty)$ with bounded derivative, and the \emph{mean Liouville-type transport equation}
\begin{equation}
    \label{eqliouvilleintro}
    \int_H \dual{\bfF(\bfu), \bfD\Psi(\bfu)}_{V', V} \;\rmd\mu(\bfu)=0,
\end{equation}
for every test functional $\Psi$ in a suitable space $\calT^\cyl$ of \emph{cylindrical test functions} (see \cref{defscyl}, in \cref{secsetting}).

The Foias-Prodi stationary statistical solutions turn out to be supported on a bounded set of the phase space (see \cref{lemmusupportedinH}). Local attractors also support some of these solutions. With such sets $B$ in mind, we consider the spaces
\begin{equation}
    \calP_\fp(B_\rmw) = \left\{ \mu \in \mathcal{P}(B_\rmw); \; \mu \textrm{ satisfying \eqref{finitemeanenstrophyintro} and \eqref{meandissipationineqsintro}} \right\},
\end{equation}
and
\begin{equation}
    \calP_\fpsss(B_\rmw) = \left\{ \mu \in \mathcal{P}(B_\rmw);\; \mu \textrm{ satisfying \eqref{finitemeanenstrophyintro}, \eqref{meandissipationineqsintro}, and \eqref{eqliouvilleintro}} \right\}.
\end{equation}

For the two-dimensional NSE \cite{RosaTemam2022}, or similarly in the finite-dimensional ODE case \cite{TobGolDoe2018}, the optimal formula contains only single points, instead of convex combinations of two points, and the proof goes essentially as follows:
\begin{align*}
    & \max_{\mu\in\calP_\fpsss(B)} \int_H \phi \;\rmd\mu \\
    & \qquad {= \max_{\mu\in\calP_\fpsss(K)}  \int_H \phi \;\rmd\mu} & & {\text{\scriptsize(supported on compact set $K$)}} \\
    & \qquad {= \sup_{\mu\in\calP_\fpsss(K)} \inf_{\Psi\in\calT^\cyl} \int_H \{\phi + \dual{\bfF, \bfD\Psi}_{V', V} \}\;\rmd\mu} & & {\text{\scriptsize(add auxiliary functional)}}\\
    & \qquad {= \sup_{\mu\in\mathcal{P}(K)} \inf_{\Psi\in\calT^\cyl} \int_H \{\phi + \dual{\bfF, \bfD\Psi}_{V', V} \}\;\rmd\mu} & & {\text{\scriptsize(extend to arbitrary measures)}}\\
    & \qquad {= \inf_{\Psi\in\calT^\cyl} \sup_{\mu\in\mathcal{P}(K)} \int_H \{\phi + \dual{\bfF, \bfD\Psi}_{V', V}\} \;\rmd\mu} & & { \text{\scriptsize(minimax theorem)}}\\
    & \qquad {= \inf_{\Psi\in\calT^\cyl} \max_{\bfu\in K} \;\{\phi + \dual{\bfF, \bfD\Psi}_{V', V} \}}, & & {\text{\scriptsize(maximum at Dirac delta)}}
\end{align*}
where $K$ is a compact set in $H$ which attracts the orbits starting from the set $B$.

This result has been proved in \cite{RosaTemam2022} for $\phi$ strongly continuous in $H$, when the forcing term $\bff$ is in $V',$ but the same result holds for $\phi$ strongly continuous in $V$, with the same assumption on $\bff,$ or even in $D(A),$ if the forcing term is assumed in $H$.

A straightforward adaptation to the 3D NSE does not seem possible, at first sight, for a number of reasons: no absorbing or attracting compact set $K$ in $H$, only on the weak topology; not enough regularity for the integrand with respect to the weak topology of $H$; and not enough regularity for a corresponding characterization of stationary statistical solutions as the measures in $\mathcal{P}(K)$ that satisfy the Liouville equation.

However, with a more judicious use of the available regularity and a more meticulous choice of the involved spaces, we are able to prove the optimal formula \eqref{optminimaxintro}, with the following steps:
\begin{align*}
    & \max_{\mu\in\calP_\fpsss(B)} \int_H \phi \;\rmd\mu & & \\
    & \quad\;\; {= \max_{\mu\in\calP_\fpsss(B_\rmw)} \int_H \phi \;\rmd\mu} & & {\text{\scriptsize(change to weak topology)}} \\
    & \quad\;\; {= \max_{\mu\in\calP_\fpsss(B_\rmw)} \inf_{\Psi\in\calT^\cyl} \int_H \{\phi + \dual{\bfF, \bfD\Psi}_{V', V} \}\;\rmd\mu} & & {\text{\scriptsize(add auxiliary functional)}}\\
    & \quad\;\; {= \max_{\mu\in\calP_\fp(B_\rmw)} \inf_{\Psi\in\calT^\cyl} \int_H \{\phi + \dual{\bfF, \bfD\Psi}_{V', V} \}\;\rmd\mu} & & {\text{\scriptsize(extend to regular measures)}}\\
    & \quad\;\; {= \inf_{\Psi\in\calT^\cyl} \max_{\mu\in\calP_\fp(B_\rmw)} \int_H \{\phi + \dual{\bfF, \bfD\Psi}_{V', V}\} \;\rmd\mu} & & { \text{\scriptsize(Sion's minimax theorem)}}\\
    & \quad\;\; {= \inf_{\Psi\in\calT^\cyl} \max_{\mu \in \Ex(\calP_\fp(B_\rmw))} \int_H \left\{ \phi + \dual{\bfF, \bfD\Psi}_{V', V}\right\}\;\rmd\mu} & & { \text{\scriptsize(Bauer's maximum principle)}}\\
    & \quad\;\; {= \inf_{\Psi\in\calT^\cyl} \max_{(\bfu_1, \bfu_2, \theta_1, \theta_2)\in \calB_f} \;\sum_{i=1}^2\theta_i\{\phi_i + \dual{\bfF_i, \bfD\Psi_i}_{V', V} \}}, & & {\text{\scriptsize(extremes in $\calP_\fp(B_\rmw)$)}}\\
\end{align*}
where $\Ex(\calP_\fp(B_\rmw))$ are the extreme points in $\calP_\fp(B_\rmw),$ which turn out to be supported on at most two Dirac delta measures, $\mu=\theta_1\delta_{\bfu_1} + \theta_2\delta_{\bfu_2},$ with $\bfu_1, \bfu_2, \theta_1, \theta_2$ in the set $\calB_f$ given in \cref{propextremePfp}, and where for notational simplicity we wrote $\phi_i = \phi(\bfu_i),$ $\bfF_i = \bfF(\bfu_i)$ and $\bfD\Psi_i = \bfD\Psi(\bfu_i).$

The first main challenge was to realize that we could properly frame the problem with the spaces and topologies above. The technicalities are quite more involved, due to the use of the weak topology and the necessary conditions defining the spaces $\calP_\fp(B_\rmw)$ and $\calP_\fpsss(B_\rmw),$ but once the proper framework is found and the topological properties of these spaces and of the functions involved are established, the proof of the final result falls into place elegantly, in a way similar to the well-posed case.

The most remarkable difference is the appearance of the convex combination of two points in the phase space, in the final minimax problem. This is actually a natural formulation given the fact that the Foias-Prodi stationary statistical solutions involve not only the stationary Liouville-type transport equation \eqref{eqliouvilleintro} but also the mean energy dissipation inequalities \eqref{meandissipationineqsintro}. These dissipation inequalities, however, form a continuum of linear constraints. Fortunately, they result in the extreme points being measures carried on a single energy shell and where the constraints reduce to a single linear constraint. In a way akin to finding the extreme points in a convex set under linear constraints, this single linear constraint on the extreme measures in $\calP_\fp(B_\rmw)$ implies that they are either Dirac delta measures or convex combinations of two Dirac delta measures, leading to the above characterization.

The mean energy dissipation inequalities \eqref{meandissipationineqsintro} are related to the energy inequality that defines Leray-Hopf weak solutions of the three-dimensional Navier-Stokes equations. If one considers instead a more general notion of ``weak'' stationary statistical solutions, satisfying only \eqref{finitemeanenstrophyintro} and \eqref{eqliouvilleintro}, then there is no linear constraint but there is no sufficient regularity for the other steps.

Concerning the steps delineated above, the first one is natural since $B$ is compact in the weak topology and this topology is good for the weak solutions of the 3D NSE. The subsequent identity follows immediately from the Liouville-type equation \eqref{eqliouvilleintro}, as in the well-posed case, so that adding the auxiliary term changes nothing. The extension to measures in $\calP_\fp(B_\rmw)$ follows from the fact that, for measures in the complementary space $\calP_\fp(B_\rmw) \setminus\calP_\fpsss(B_\rmw)$, the infimum in $\Psi$ of the integral expression is $-\infty,$ so they do not contribute to the maximum in $\mu$. This is, again, similar to the 2D case, except for the choice of extending the maximum to $\calP_\fp(B_\rmw)$ instead of the larger space $\calP(B_\rmw)$. One important fact here is that the space $\calP_\fp(B_\rmw)$ is a closed convex subspace of $\mathcal{P}(B_\rmw)$, hence it is also compact. The most distinguishing fact, however, is that the measures in $\calP_\fp(B_\rmw)$ turn out to be uniformly tight with respect to the bounded sets in $V,$ which is compactly included in $H$. This latter fact guarantees that $G(\mu, \Psi) = \int_H \{\phi + \dual{F, \bfD\Psi} \}\;\rmd\mu$ is continuous with respect to $\mu\in \calP_\fp(B_\rmw)$. This map is also continuous with respect to $\Psi\in \calT^\cyl,$ for a suitable norm in $\calT^\cyl$. These continuity properties of $G(\mu, \Psi)$ and the compactness of $\calP_\fp(B_\rmw)$ allow us to apply Sion's minimax theorem to switch the order of $\max$ and $\inf$. The continuity of $G(\mu, \Psi)$ with respect to $\mu \in \calP_\fp(B_\rmw)$ and the compactness of this space also allow us to apply Bauer's maximum principle to deduce that the maximum in $\mu$ is achieved on extreme points. Finally, the structure of $\calP_\fp(B_\rmw)$, with the constraints \eqref{meandissipationineqsintro}, allows us to deduce first that the extreme measures are supported on an energy shell $\{\|\bfu\|_{L^2}^2 = e\}$ and then, that these constraints reduce to a single linear constraint on that energy shell, yielding that the extreme points are convex combinations of at most two Dirac delta measures carried by the set $B_\rmw$. This convex combination belongs to $\calP_\fp(B_\rmw)$, which is equivalent to the restriction to $\calB_f.$

The result \eqref{optminimaxintro} is established for any functional $\phi$ which is weakly continuous on $V$ and with sub-quadratic growth in the norm of $V$ (see \cref{lemgvwGcontinuous}). This includes the kinetic energy 
\[ \phi(\bfu) = \frac{\rho}{2} \int_\Omega |\bfu(\bfx)|^2 \;\rmd\bfx,
\]
and the energy fluxes
\[ \phi(\bfu) = \rho \int_\Omega (\bfu(\bfx) \cdot \bfnabla)\bfu(\bfx) \cdot (P_k \bfu)(\bfx) \;\rmd\bfx,
\]
and
\[
    \phi(\bfu) = \rho \int_\Omega (\bfu(\bfx) \cdot \bfnabla)\bfu(\bfx) \cdot (Q_k \bfu)(\bfx) \;\rmd\bfx,
\]
where $\rho$ is the density of the fluid, assumed homogeneous, and $P_k$ is the Galerkin projector, i.e. the orthogonal projector in $H$ onto the span of the first $k$ eigenmodes of the Stokes operator, while $Q_k = I - P_k.$ It does not include the enstrophy, nor the energy dissipation rate, which involve the integrand $|\bfnabla\bfu(\bfx)|^2.$ For these, we need to restrict the stationary statistical solutions to those with a stronger bound, valid, for instance, for any stationary statistical solution obtained as a generalized limit of time-averages of Leray-Hopf weak solutions (see \cref{secvfsss}).

Another approach is to approximate the functional $\phi(\bfu) = \|\bfnabla\bfu\|_{L^2}^2$ from below by $\phi_k(\bfu) = \|\bfnabla P_k\bfu\|_{L^2}^2$ and apply the optimal minimax upper bound formula to $\phi_k$, which yields an optimal upper bound formula for $\phi$ as the limit of minimax problems, as discussed in \cref{secenstrophyupperbound}.

On a different direction, by switching the sign of $\phi,$ we can also obtain an optimal lower bound formula. Indeed, by applying the optimal minimax upper bound formula to $-\phi$ and using that $\calT^\cyl$ is a vector space, we obtain a corresponding optimal lower bound formula
\[
    \min_{\mu\in\calP_\fpsss(B)} \average{\phi}_\mu = \sup_{\Psi\in\calT^\cyl} \min_{(\bfu_1, \bfu_2, \theta_1, \theta_2)\in\calB_f} \;\sum_{i=1}^2\theta_i\{\phi(\bfu_i) + \dual{\bfF(\bfu_i), \bfD\Psi(\bfu_i)}_{V', V} \}.
\]

Moreover, both Sion's minimax theorem and Bauer's maximum principle hold more generally for upper-semi-continuous functions $\mu \mapsto G(\mu, \Psi)$, so that we can actually apply the minimax formula to $\phi(\bfu) = - \|\bfnabla\bfu\|_{L^2}^2,$ which corresponds to an optimal lower bound formula for the enstrophy and energy dissipation rate.

The core of the paper is organized as follows. In \cref{secsetting}, we recall the mathematical setting for the three-dimensional Navier-Stokes equations and the proper functional-analytic framework for the Foias-Prodi stationary statistical solutions. In \cref{secbackground}, we recall known results related to the topology of probability spaces, Sion's minimax theorem, extreme points, and Bauer's maximum principle. In \cref{secprelim}, we establish some preliminary results needed for the steps of the proof delineated above. In \cref{secmainresult}, we prove the optimal minimax formula. Finally, in \cref{seconclusions}, we present a number of concluding remarks and perspectives for future works.

\section{Mathematical setting}
\label{secsetting}

In this section, we recall the Leray-Hopf framework for three-dimensional NSE on bounded domains and the proper functional-analytic formulation for the Foias-Prodi stationary statistical solutions.

\subsection{The three-dimensional Navier-Stokes equations}
\label{sec3dnse}

Here we recall some fundamental and classical results about the individual solutions of the three-dimensional Navier-Stokes equations, which can be found, for instance, in \cite{ConstFoias1989,FMRT2001,Lady1963,Temam1984,Temam1988,Temam1995}.

We consider the flow of a three-dimensional homogeneous incompressible Newtonian viscous fluid enclosed in a bounded domain $\Omega\subset\RR^3$ with boundary $\partial\Omega$ and governed by the Navier-Stokes equations, with no-slip boundary conditions. We denote by $u(t, \bfx)\in\RR^3$ and $p(t, \bfx)\in\RR$, respectively, the velocity and the kinematic pressure of the fluid at a point $\bfx=(x_1, x_2, x_3)\in\Omega$ and time $t\ge0$, with velocity coordinates $\bfu = (u_1, u_2, u_3).$ The velocity and pressure are thus determined by the following initial-boundary value problem:
\begin{equation}
  \label{nse}
  \begin{cases}
	{\displaystyle\frac{\partial \bfu}{\partial t}}
	 - \nu \Delta \bfu + (\bfu\cdot\bfnabla)\bfu + \bfnabla p = \bff,
		\qquad\text{in } \Omega,
	\\
	\bfnabla \cdot \bfu = 0, \qquad\text{in } \Omega,
	\\
	\bfu=\bfzero, \qquad\text{on } \partial\Omega,
  \end{cases}
\end{equation}
where $\nu>0$ is the kinematic viscosity of the fluid, and $\bff=\bff(\bfx)\in\RR^3$ is the external mass density of body forces, which is assumed to be time independent.

We let $L^2(\Omega)^3$ and $H_0^1(\Omega)^3$ denote the usual Lebesgue and Sobolev spaces of three-dimensional vector fields on $\Omega.$ We denote the norm in $L^2(\Omega)^3$ by $\|\cdot\|_{L^2}.$ Since $\Omega$ is assumed to be bounded, the norm in $H_0^1(\Omega)^3$ is taken to be $\|\bfnabla \cdot \|_{L^2},$ with the understanding that $\bfnabla \bfu = \bfnabla \otimes \bfu = (\partial_{x_j}u_i)_{i,j=1}^3.$

Consider the space $H$ defined as the $L^2$-closure of the space of smooth, compactly supported divergence-free vector fields, i.e. the closure of $\{\bfu\in \calC_\rmc^\infty(\Omega)^3; \;\bfnabla \cdot \bfu = 0\}$ in $L^2(\Omega)^3.$ The corresponding closure in $H_0^1(\Omega)^3$ is denoted by $V.$ The norms and inner products on $H$ and $V$ are those inherited from $L^2$ and $H_0^1$, respectively.

The spaces $H$ and $V$ are two separable Hilbert spaces, with $V$ densely and compactly immersed into $H$. We identify $H$ with its dual, so that $V \subset H = H' \subset V'$. The duality product of the pair $V', V$ is denoted by $\dual{\cdot, \cdot}_{V', V}.$ The induced norm in $V'$ is denoted by $\|\cdot\|_{V'}$.

The space $H$ with its weak topology $\sigma(H, H')$ is denoted by $H_\rmw$. Any subset $B\subset H$ endowed with the weak topology inherited from $H$ is denoted $B_\rmw.$ With that, we denote the space of real-valued functions on $H$ which are continuous under the weak topology of $H$ by $\calC(H_\rmw)$, and the subspace of those that are bounded is denoted $\calC_\rmb(H_\rmw)$. In order to avoid any confusion, the weak topology in $V$ is simply referred to by $\sigma(V, V')$.

We have the Stokes operator $A:D(A)\subset H \rightarrow H$ which is a positive self-adjoint operator densely defined in $H$ and with a compact inverse. Its eigenvalues are positive and ordered in nondecreasing order and repeated accordingly to their multiplicity. The first eigenvalue is denoted by $\lambda_1.$ The Galerkin projectors onto the first $k$ eigenmodes are denoted by $P_k:H\rightarrow H.$ We also have $V = D(A^{1/2}).$ 

We finally define the function
\begin{equation}
    \label{defF}
    \bfF = \bff - \nu A\bfu - B(\bfu, \bfu),
\end{equation}
where $B(\cdot, \cdot): V\times V \rightarrow V'$ is the bilinear operator defined by duality via the trilinear form, \[\dual{B(\bfu, \bfv),\bfw}_{V',V}=b(\bfu, \bfv, \bfw) = \int_\Omega (\bfu \cdot \bfnabla)\bfv \cdot \bfw \;\rmd\bfx, \quad \forall \,\bfu,\bfv,\bfw\in V.\]
We assume that $\bff\in V'$, so that $\bfF:V\rightarrow V'$ is continuous from $V$ into $V'.$ In this framework, the 3D NSE can be written as the functional equation $\bfu_t = \bfF(\bfu),$ in $L^1(0, T; V'),$ for $T>0.$ Given an initial condition $\bfu_0\in H,$ there exists at least one solution $\bfu \in \calC([0, T], H_\rmw) \cap L^2(0, T; V)$ with $\bfu_t \in L^{4/3}(0, T; V')$. The following estimate holds
\begin{equation}
    \label{eqenergyineq}
    \|\bfu(t)\|_{L^2}^2 \leq \|\bfu(0)\|_{L^2}^2 e^{-\nu \lambda_1 t} + \frac{1}{\nu^2\lambda_1} \|\bff\|_{V'}^2 \left( 1 - e^{-\nu \lambda_1 t}\right), \quad \forall t \geq 0.
\end{equation}

Define the radius
\begin{equation}
    \label{defrho0}
    \radius_0 = \frac{\|\bff\|_{V'}}{\nu\lambda_1^{1/2}}.
\end{equation}
Inequality \eqref{eqenergyineq} implies that, for any $\radius > \radius_0,$ the closed ball $B_H(\radius) = \{\bfu \in H; \; \|\bfu\|_{L^2} \leq \radius \}$ in $H$ with radius $\radius$ is absorbing, i.e. any weak solution $\bfu=\bfu(t)$ enters, and remains, in $B_H(\radius)$ in finite time.

We also need the following explicit bound on $B(\bfu, \bfu)$ in $V'$ (see \cite[Chapter II, eq. (A.31)]{FMRT2001})
\begin{equation}
    \label{eqboundBuuinVprime}
    \|B(\bfu, \bfu)\|_{V'} \leq c_L \|\bfu\|_{L^2}^{1/2} \|\bfnabla\bfu\|_{L^2}^{3/2},
\end{equation}
where $c_L$ is a universal constant (from the $L^4$ Ladyzhenskaya inequality in three dimensions).

We say that a subset $B\subset H$ is \textbf{positively invariant} for the three-dimensional Navier-Stokes equations if, for any initial condition $\bfu_0\in B,$ we have that any weak solution $\bfu=\bfu(t)$ with $\bfu(0) = \bfu_0$ is such that $\bfu(t)\in B,$ for any $t\geq 0.$ Of course, $H$ itself is positively invariant, as well as the ball $B_H(\radius)$, for any $\radius \geq \radius_0.$

\subsection{Test functionals for the stationary statistical solutions}
The stationary statistical solutions are defined in a weak sense, and for that we need a notion of test functional, as follows.

\begin{definition}
    \label{defscyl}
    A \textbf{cylindrical test functional} is any function $\Psi:V'\rightarrow \RR$ of the form
    \[ \Psi(\bfu) = \psi(\dual{\bfu, \bfw_1}_{V',V}, \ldots, \dual{\bfu, \bfw_m}_{V', V}), \qquad \forall \bfu\in V',
    \]
    for some $\bfw_1, \ldots, \bfw_m\in V$, $m\in \NN$, and $\psi\in \calC_c^1(\RR^m)$. We denote the set of such test functionals by $\calT^\cyl$.
\end{definition}


The cylindrical test functionals are Fr\'echet differentiable in $V'$, with differential $\bfD\Psi(\bfu),$ in $V$, given by
\begin{equation}
    \label{eqfrechetPsi}
    \bfD\Psi(\bfu) = \sum_{j=1}^m \partial_j \psi(\dual{\bfu, \bfw_1}_{V',V}, \ldots, \dual{\bfu, \bfw_m}_{V', V})\bfw_j, \quad \forall \bfu\in V'.
\end{equation}

Clearly, $\calT^\cyl$ is a vector space. It becomes a normed vector space under the norm
\begin{equation}
    \label{eqcalTcylmetric}
    \|\Psi\|_{\calT^\cyl} = \sup_{\bfu\in H} |\Psi(\bfu)| + \sup_{\bfu\in V} \|\bfnabla \bfD\Psi(\bfu)\|_{L^2}.
\end{equation}

The cylindrical test functionals are simple enough to easily allow for suitable a~priori estimates, and their set is large enough to be dense in the space of bounded continuous functionals on $H$, under the weaker norm $\sup_{\bfu\in H} |\Psi(\bfu)|.$ They are also dense, with respect to the norm \eqref{eqcalTcylmetric}, in the set of general test functionals discussed in \cite[Definition 6.2]{RosaTemam2022} and introduced earlier by Foias \cite{Foias72,Foias73}. We use the norm \eqref{eqcalTcylmetric} to have the continuity of the functional in the Liouville-type transport equation \eqref{eqliouvilleintro}.

\subsection{Stationary statistical solutions}
\label{secsss}

With the definition of cylindrical test functionals, we have the following notion of Foias-Prodi stationary statistical solution. We refer the reader to \cite{FMRT2001}.

\begin{definition}
    \label{defFPSSS}
    We define a \textbf{(Foias-Prodi) stationary statistical solution} as any Borel probability measure $\mu$ on $H$ with the properties that
    \begin{subequations}
      \begin{align}
        \label{finitemeanenstrophy}
        (i) & \int_H \|\bfnabla \bfu\|_{L^2}^2\;\rmd\mu(\bfu) < \infty, \\
        \label{meandissipationineqs}
        (ii) & \int_H \psi'(\|\bfu\|_{L^2}^2)\left( \nu \|\bfnabla \bfu\|_{L^2}^2 - \dual{\bff, \bfu}_{V', V} \right) \;\rmd\mu(\bfu) \leq 0, \\
        \label{eqliouville}
        (iii) & \int_H \dual{\bfF(\bfu), \bfD\Psi(\bfu)}_{V',V} \;\rmd\mu(\bfu) = 0, \quad \forall \Psi\in \calT^\cyl,
      \end{align}
    \end{subequations}
    for any real-valued continuously-differentiable function $\psi$ on $[0, \infty)$ which is nonnegative, non-decreasing, and with bounded derivative.
\end{definition}

A particular family of stationary statistical solution is obtained via the generalized limit of time averages of Leray-Hopf weak solutions. Indeed, for any global Leray-Hopf weak solution $\bfu = \bfu(t),$ $t \geq 0,$ for which $\bfu \in \calC([0, T], H_\rmw) \cap L^2(0, T; V)$, for any $T > 0,$ the time average
\[
    \frac{1}{T}\int_0^T \varphi(\bfu(t))\;\rmd t
\]
is a bounded function of $T > 0,$ for any bounded Borel-measurable functional $\varphi$ on $H$. Then, for any generalized limit $\Lim$ on the space of bounded functions on $[0, \infty)$, we obtain a positive linear functional which is represented by a Borel probability measure $\mu=\mu_\bfu$ (it also depends on the choice of $\Lim$, but we omit this dependence for notational simplicity), i.e.
\[
    \int_H \varphi(\bfv)\;\rmd\mu_{\bfu}(\bfv) = \Lim \frac{1}{T}\int_0^T \varphi(\bfu(t))\;\rmd t, \quad \forall \varphi \in \calC_\rmb(H_\rmw).
\]
This measure $\mu_{\bfu}$ has all the properties described in \cref{defFPSSS}. A stationary statistical solution obtained as the generalized limit of time averages of a weak solution is called a \textbf{time-average stationary statistical solution}. More details about such solutions can be found in \cite[Section IV.3.1]{FMRT2001}.

If the initial condition $\bfu(0)$ belongs to a positively invariant and weakly compact set $B\subset H,$ then $\bfu(t)\in B$ for all $t\geq 0$ and the corresponding measure $\mu_\bfu$ is carried by $B.$ Thus, we find that, in any positively invariant and weakly compact set $B\subset H,$ there exists at least one time-average stationary statistical solution which is carried by $B$. This applies to local basins of attraction but also to unstable fixed points and $\alpha$-limit sets.



\section{Background results}
\label{secbackground}

In this section, we recall some known results in measure theory and in functional analysis that relate to our mathematical framework.

\subsection{Measure spaces}

The space of Borel probability measures on a topological space $X$ is denoted by $\calP(X)$. We endow it with the standard weak-star topology, which is the smallest topology for which $\mu \mapsto \int_X \varphi(u) \;\rmd\mu(u)$ is continuous, for every $\varphi\in \calC_\rmb(X),$ which denotes the space of bounded, continuous real-valued functions on $X$.  We denote the weak-star convergence of a net $(\mu_\alpha)_\alpha$ to a measure $\mu$ in this topology by $\mu_\alpha \wconv \mu,$ in $\calP(X)$, meaning that $\int_X \varphi(u) \;\rmd\mu_\alpha(u) \rightarrow \int_X \varphi(u) \;\rmd\mu(u),$ for all $\varphi\in\calC_\rmb(X).$ 

A measure $\mu\in\calP(X)$ is called \textbf{tight} when, for every Borel subset $E\subset X$ and every $\varepsilon > 0,$ there exists a compact subset $K\subset E$ such that $\mu(E\setminus K) < \varepsilon.$ We denote by $\calP_\tau(X)$ the space of tight Borel probability measures on $X$, endowed with the topology inherited from $\calP(X).$

A family $\calF$ of finite Borel measures on $X$ is called \textbf{(uniformly) tight} when, for every $\varepsilon > 0,$ there exists a compact subset $K\subset X$ such that $\mu(X\setminus K) < \varepsilon,$ for all $\mu\in\calF.$ When $X$ is a separable metrizable space, then every uniformly tight family of measures in $\calP(X)$ is relatively compact in $\calP(X)$ \cite[Theorem 12.5]{AliBor2006}. When $X$ is metrizable and compact, every Borel probability measure on $X$ is tight, so $\calP(X) = \calP_\tau(X)$, and, since the whole space $X$ is compact, the space $\calP(X)$ is uniformly tight and hence compact (see also \cite[Theorem 12.11]{AliBor2006}). If $X$ is a Polish space (a separable and complete metrizable space), then any Borel probability measure is tight \cite[Theorem 12.7 with Lemma 12.6]{AliBor2006}, so that $\calP(X) = \calP_\tau(X)$. Still in this case that $X$ is a Polish space, Prokhorov's Theorem (see e.g. \cite{Prokhorov1956} and \cite[Theorem 11.5.4]{Dudley2002}) asserts that uniform tightness and relative compactness are actually equivalent, i.e. a family $\calF$ of Borel probability measures on $X$ is uniformly tight \emph{if and only if} it is relatively compact in $\calP(X)$.

When $X$ is completely regular, it follows from the Portmanteau Theorem \cite[Theorem 8.1]{Topsoe1970} that the weak-star convergence in the space $\calP_\tau(X)$ of tight measures is equivalent to the lower-semi-continuous weak-star convergence, i.e. it is equivalent to $\int_X \varphi(u) \;\rmd\mu(u) \leq \liminf_{\alpha}\int_X \varphi(u) \;\rmd\mu_\alpha(u),$ for all bounded and lower-semi-continuous real-valued functions on $X.$ 

For a Borel subset $E\subset X$, we denote by $\calP(E)$ the subspace of $\calP(X)$ of measures $\mu\in \calP(X)$ which are carried by $E,$ i.e. $\mu(X\setminus E) = 0,$ endowed with the topology inherited from $\calP(X)$. Similarly for $\calP_\tau(E).$ When $X$ is completely regular, the topology of $\calP_\tau(E)$ inherited from $\calP_\tau(X)$ coincides with the topology on $\calP_\tau(E)$ when we view $E$ itself as a topological space. This is also a consequence of the Portmanteau Theorem \cite[Theorem 8.1]{Topsoe1970}.\footnote{Indeed, as a subspace, $\mu_\alpha \wconv \mu$ if $\int_X \varphi(u)\;\rmd\mu_\alpha(u) \rightarrow \int_X \varphi(u)\;\rmd\mu(u)$, for every $\varphi\in\calC_\rmb(X),$ while as a topological space by itself, $\mu_\alpha\wconv \mu$ if the convergence holds for every $\varphi\in\calC_\rmb(E).$ Since $\varphi\in\calC_\rmb(X)$ implies $\varphi|_E \in \calC_\rmb(E)$, the latter convergence implies the former. For the converse, we use the Portmanteau Theorem \cite[Theorem 8.1]{Topsoe1970}, which says that, for tight measures on a completely regular space, the former convergence is equivalent to $\limsup_\alpha \mu_\alpha(C) \leq \mu(C),$ for every closed subset $C$ in $X$. Now, if $F$ is a relatively closed set in $E$, by definition it is of the form $F=C\cap E,$ for some $C$ closed in $X.$ Then, since $\mu_\alpha$ is carried by $E$, we have $\mu_\alpha(C) = \mu_\alpha(C\cap E) + \mu_\alpha(C\setminus E) = \mu_\alpha(F)$. Similarly, $\mu(C) = \mu(F).$ Thus, $\lim_\alpha \mu_\alpha(F) = \lim_\alpha \mu_\alpha(C) \leq \mu(C) = \mu(F),$ showing that $\mu_\alpha \wconv \mu$ in the latter sense, completing the proof of equivalence.} Thus, at least in this case, there is no ambiguity in the notation $\calP_\tau(E)$, and we have that the convergence $\mu_\alpha \wconv \mu$ in $\calP_\tau(E)$ means that $\int_X \varphi(u) \;\rmd\mu_\alpha(u) \rightarrow \int_X \varphi(u) \;\rmd\mu(u),$ for all $\varphi\in\calC_\rmb(X)$, or, equivalently, for all $\varphi\in\calC_\rmb(E).$ The property of equivalence between the weak-star convergence and the lower-semi-continuous weak-star convergence also carries over to the subspaces $\calP_\tau(E)$.

Another useful consequence of the Portmanteau Theorem \cite[Theorem 8.1]{Topsoe1970}, when $X$ is completely regular, is that, if $E$ is a closed subset of $X$, then $\calP_\tau(E)$ is a closed subset of $\calP_\tau(X).$

In what follows, $X$ will be either $H$ or $H_\rmw$, so we turn our attention to each of these spaces. Both are completely regular (see \cite[Footnote 2, Section 2.2]{FRT2019}) and, as we will see, any Borel probability measure on them is tight, so all the above applies.

Concerning the space $H$, since it is metrizable and separable, the space $\calP(H)$ is also metrizable and separable \cite[Theorem 15.12]{AliBor2006}. In particular, topological properties can be characterized by limits of sequences instead of nets. The space $H$ is, in fact, a Polish space (a separable and complete metrizable space), so that, as discussed above,  $\calP(H) = \calP_\tau(H)$. This means the weak-star convergence in this space is equivalent to the lower-semi-continuous weak-star convergence. Moreover, $\calP(E) = \calP_\tau(E)$, for any Borel subset $E\subset H$, and all the properties discussed above for subspaces apply.

We also consider the space $\calP(H_\rmw)$ of Borel probability measures with respect to the weak topology of $H$. Since $H$ is a separable Hilbert space, the Borel sets for the weak topology coincide with the Borel sets for the strong topology (see \cite[Appendix IV.A.1]{FMRT2001}), so that $\calP(H_\rmw)$ coincides with $\calP(H)$ as sets, but they do not coincide as topological spaces. Since $\calC_\rmb(H_\rmw)$ is strictly contained in $\calC_\rmb(H),$ the topology in $\calP(H)$ is strictly finer than that in $\calP(H_\rmw)$. Nevertheless, since any Borel probability measure on $H_\rmw$ is a Borel probability measure on $H$  and since any Borel probability measure on $H$ is tight in $H$, and since any strongly compact set in $H$ is weakly compact, it follows that any measure on $\calP(H_\rmw)$ is tight with respect to compact subsets in the weak topology of $H$. Thus, we also have $\calP(H_\rmw) = \calP_\tau(H_\rmw)$. Moreover, using again that $H_\rmw$ is completely regular and that the space of tight measures on a completely regular space is Hausdorff, we have that $\calP(H_\rmw)$ is a Hausdorff space. We can also prove this directly using characteristic functions, but the previous argument also guarantees that there is no difference in how we choose the topology in the space of Borel probability measures carried by a Borel subspace $E_\rmw$ of $H_\rmw$. It also shows that the weak-star convergence in the space $\calP(H_\rmw)$ is equivalent to the lower-semi-continuous weak-star convergence, i.e. it is equivalent to $\int_H \varphi(u) \;\rmd\mu(u) \leq \liminf_{\alpha}\int_H \varphi(u) \;\rmd\mu_\alpha(u),$ for all bounded and lower-semi-continuous real-valued functions on $H_\rmw$. Similarly, in the space $\calP(E_\rmw)$, where $E$ is a Borel subset of $H$, the weak-star convergence is equivalent to the lower-semi-continuous weak-star convergence, i.e. weak-star convergence with respect to bounded and lower-semi-continuous real-valued functions on $E_\rmw.$ At this point, however, we do not have that topological properties in $\calP(H_\rmw)$ can be characterized by sequences of measures.

One of our main interests in working with the weak topology of $H$, however, arises when we restrict the measures to have support on a bounded, closed, convex set $B$ in $H$, which is compact for the weak topology. With $B$ weakly compact, which is to say that $B_\rmw$ is compact, we find that $\calP(B_\rmw)$ is also compact (see \cite[Theorem 15.11]{AliBor2006}), as discussed above. This is a fundamental property that we will explore on its own. It also implies that $\calP(B_\rmw)$ is metrizable and separable, since $B_\rmw$ is metrizable and separable \cite[Theorem 15.12]{AliBor2006}. This means that we can prove closedness, continuity, and other topological properties in $\calP(B_\rmw)$ by working with sequences of measures.

\subsection{Measure spaces associated with stationary statistical solutions}
\label{secmeasurespacessss}

The stationary statistical solutions form a subspace of the space of measures $\calP(H).$ Here we define this subspace and other auxiliary subspaces with some of the conditions for being a stationary statistical solution.

\begin{definition}
    We define 
    \begin{equation}
        \calP_\fpsss(H) = \left\{ \mu \in \mathcal{P}(H);\; \mu \textrm{ satisfying \eqref{finitemeanenstrophy}, \eqref{meandissipationineqs}, and \eqref{eqliouville}} \right\}
    \end{equation}
    as the subspace of $\calP(H)$ composed of all the Foias-Prodi stationary statistical solutions, endowed with the topology inherited from $\calP(H)$, and denote by $\calP_\fpsss(H_\rmw)$ the corresponding space endowed with the weak-star topology inherited from $\calP(H_\rmw).$ We also define the subspaces
    \[ 
        \calP_\fpsss(B) = \calP_\fpsss(H) \cap \calP(B), \qquad \calP_\fpsss(B_\rmw) = \calP_\fpsss(H_\rmw) \cap \calP(B_\rmw),
    \]
    for any Borel subset $B\subset H.$
\end{definition}

We have seen at the end of \cref{secsss} that any weakly compact set in $H$ that is positively invariant carries at least one time-average measure. We have also seen, at the end of \cref{sec3dnse}, that $B_H(\radius)$ is positively invariant, for any $\radius \geq \radius_0$, and it is also weakly compact. Thus, it follows that 
\begin{equation}
    \label{lemPBHr0notempty}
    \calP_\fpsss(B_\rmw) \neq \emptyset,
\end{equation}
for $B = B_H(\radius)$, for any $\radius \geq \radius_0$, or for any other $B$ which is positively invariant and weakly compact in $H$.

Note, however, that the invariance of $B$ is not required for the existence of stationary statistical solutions within $B$. Indeed, any unstable fixed point or $\alpha$-limit set give rise to stationary statistical solutions, and $B$ may be a neighborhood of such sets and not invariant (c.f. \cref{secmaxasymptimeavg}).

Besides the spaces of stationary statistical solutions, we will also need the following auxiliary spaces.

\begin{definition}
    Define 
    \begin{equation}
        \calP_\fp(H) = \left\{ \mu \in \mathcal{P}(H);\; \mu \textrm{ satisfying \eqref{finitemeanenstrophy}, \eqref{meandissipationineqs}} \right\}
    \end{equation}
    as the subspace of $\calP(H)$ of measures $\mu$ satisfying the finite-mean-enstrophy condition \eqref{finitemeanenstrophy} and the mean energy dissipation inequalities \eqref{meandissipationineqs}, with $\calP_\fp(H_\rmw)$ as the corresponding space endowed with the topology inherited from $\calP(H_\rmw).$ We also define $\calP_\fp(B) = \calP_\fp(H) \cap \calP(B)$ and $\calP_\fp(B_\rmw) = \calP_\fp(H_\rmw) \cap \calP(B_\rmw),$ for any Borel subset $B\subset H$.
\end{definition}

\begin{definition}
    Let
    \begin{equation}
        \label{defLmu}
        L(\mu) = \int_H \ell(\bfu) \;\rmd\mu(\bfu),
    \end{equation}
    with
    \begin{equation}
        \label{defellmu}
        \ell(\bfu) = \nu \|\bfnabla \bfu\|_{L^2}^2 - \dual{\bff, \bfu}_{V', V}.
    \end{equation}
    Define $\calP_\rmd(H)$ as the subspace of $\calP(H)$ of measures $\mu$ with finite enstrophy and satisfying the dissipation bound
    \begin{equation}
        \label{meancleandissipationbound}
        L(\mu) \leq 0,
    \end{equation}
    i.e.
    \begin{equation}
        \calP_\rmd(H) = \left\{\mu\in\calP(H); \mu \textrm{ satisfying \eqref{finitemeanenstrophy}, \eqref{meancleandissipationbound}} \right\},
    \end{equation}
    with $\calP_\rmd(H_\rmw)$ as the corresponding space endowed with the topology inherited from $\calP(H_\rmw).$
    We also define $\calP_\rmd(B) = \calP_\rmd(H) \cap \calP(B)$ and $\calP_\rmd(B_\rmw) = \calP_\rmd(H_\rmw) \cap \calP(B_\rmw),$ for any Borel subset $B\subset H$.
\end{definition}

\begin{definition}
    Define $\calP_\ens(H)$ as the subspace of $\calP(H)$ of measures $\mu$ satisfying the mean-enstrophy bound
    \begin{equation}
        \label{meanenstrophybound}
        \int_H \|\bfnabla \bfu\|_{L^2}^2\;\rmd\mu(\bfu) \leq \frac{\|\bff\|_{V'}^2}{\nu^2},
    \end{equation}
    i.e.
    \begin{equation}
        \calP_\ens(H) = \left\{ \mu \in \mathcal{P}(H);\; \mu \textrm{ satisfying \eqref{meanenstrophybound}} \right\},
    \end{equation}
    with $\calP_\ens(H_\rmw)$ as the corresponding space endowed with the topology inherited from $\calP(H_\rmw).$
    We also define $\calP_\ens(B) = \calP_\ens(H) \cap \calP(B)$ and $\calP_\ens(B_\rmw) = \calP_\ens(H_\rmw) \cap \calP(B_\rmw),$ for any Borel subset $B\subset H$.
\end{definition}

\begin{definition}
    Define $\calP_\rmv(H)$ as the subspace of $\calP(H)$ of measures $\mu$ with finite mean-enstrophy, i.e.
    \begin{equation}
        \calP_\rmv(H) = \left\{ \mu \in \mathcal{P}(H);\; \mu \textrm{ satisfying \eqref{finitemeanenstrophy}} \right\},
    \end{equation}
    with $\calP_\rmv(H_\rmw)$ as the corresponding space endowed with the topology inherited from $\calP(H_\rmw).$
    We also define $\calP_\rmv(B) = \calP_\rmv(H) \cap \calP(B)$ and $\calP_\rmv(B_\rmw) = \calP_\rmv(H_\rmw) \cap \calP(B_\rmw),$ for any Borel subset $B\subset H$.
\end{definition}

These spaces are related by the inclusions
\[
    \calP_\fpsss(H) \subset \calP_\fp(H) \subset \calP_\rmd(H) \subset \calP_\ens(H) \subset \calP_\rmv(H) \subset \calP(H),
\]
with corresponding inclusions for the spaces based on the weak topology of $H$ and on Borel subsets $B$ of $H$. The fact that $\calP_\fp(H) \subset \calP_\rmd(H)$ follows immediately from taking $\psi(s) = s$ in the mean energy dissipation inequalities \eqref{meandissipationineqs}. The only inclusion that is not so immediate is $\calP_\rmd(H)\subset\calP_\ens(H),$ which is proved in \cref{secprelim}. The others follow directly from the definitions.

One may ask why we need so many spaces. For one, the space $\calP_\fpsss(H)$ is the main space of interest, of the Foias-Prodi stationary statistical solutions. The space $\calP_\fp(H)$ enters directly as a fundamental step in the proof of the minimax formula, via the selection principle relating these two spaces via the Liouville-type transport equation. The space $\calP_\ens(H)$ has the essential regularity for proving the continuity of many of the involved functionals, as required by Sion's and Bauer's theorems. The spaces $\calP_\rmd(H)$ and $\calP_\rmv(H)$ appear naturally in the characterization of the extreme points of $\calP_\fp(H).$

\subsection{Sion's minimax theorem}

As done in \cite{TobGolDoe2018}, we use Sion's Minimax Theorem \cite{Sion1958} to switch the order of the maximization over all measures in $\calP_\fp(B_\rmw)$ and the minimization over all test functionals.

\begin{theorem}[Sion's Minimax Theorem]
  \label{thmsionminimax}
  Let $U$ be a compact convex subset of a topological vector space and $Z$ a convex subset of a possibly different topological vector space. Let $f:U\times Z \rightarrow \RR$ be such that
  \begin{enumerate}
    \item $f(\cdot,v)$ is upper semi-continuous and quasi-concave on $U$, for each $v\in Z$;
    \item $f(u,\cdot)$ is lower semi-continuous and quasi-convex on $Z$, for each $u\in U$.
  \end{enumerate}
  Then,
  \[ \max_{u\in  U} \inf_{v\in Z} f(u,v) = \inf_{v\in Z} \max_{u\in U} f(u,v).
  \]
\end{theorem}

The statement in \cite{Sion1958} is actually given with the supremum instead of maximum, but since $U$ is compact these two versions coincide. See also the version stated in the Introduction of \cite{Komiya1988}, with the corresponding result $\min_u \sup_v g(u,v) = \sup_v \min_u g(u,v)$, for $g(u,v) = -f(u, v)$.

Given two sets $A$ and $B$, a function $f:A\times B\rightarrow\RR$ is called \textbf{quasi-concave} on $A$ if $\{u\in A; \; f(u,v)\geq c\}$ is a convex set in $A$, for any $v\in B$ and any $c\in \RR$, and is called \textbf{quasi-convex} on $B$ if $\{v\in B; \; f(u,v)\leq c\}$ is a convex set in $B$, for any $u\in A$ and any $c\in \RR$.

In our case, we apply this minimax result to functions which are linear in one variable and affine in the other. If $X$ is a vector space and $A\subset X$ is convex, then a function $f:A\rightarrow \RR$ is called \textbf{affine} when $f(\theta u +(1-\theta)v) = \theta f(u) + (1-\theta) f(v)$, for all $u, v\in A$ and all $0\leq \theta \leq 1$. As such, the function is quasi-concave and quasi-convex on either variable, so that we only need to worry about continuity of $f$ and the compactness of $U$.

\subsection{Extreme points}

An \textbf{extreme point} of a subset $A$ of a real vector space $X$ is any point in $A$ that is not in between two other points in $A$, i.e. $x$ is an extreme point of $A$ if it cannot be written as $x = (1 - \theta) y + \theta z,$ for $y, z \in A$, $y\neq z,$ and $0 < \theta < 1.$ We denote the subset of extreme points of a set $A$ by
\begin{equation}\label{def:extreme:set}
    \Ex(A) = \{x\in A; \; x \textrm{ is an extreme point of } A\}.
\end{equation}

We are particularly interested in extreme points of convex sets. In relation to that, a non-empty  subset $F$ of a convex set $C$ is called a \textbf{face} of $C$ if for any $y, z \in C$ and $0 < \theta < 1$ such that $(1-\theta)y + \theta z \in F,$ then $y, z\in F$ (see \cite[Chapter 8]{Simon2011} or \cite[Section 7.12]{AliBor2006}). For example, we prove in \cref{lemPdfaceofP} that the set of probability measures with finite enstrophy is a face of the convex set of probability measures. Extreme points of faces of convex sets are easy to relate to extreme points of the convex set. Indeed, if $F$ is a face of a convex set $C,$ then $\Ex(F) = \Ex(C) \cap F.$\footnote{If $x\in \Ex(C) \cap F,$ it is immediate to deduce, since $F\subset C,$ that $x\in\Ex(F).$ Conversely, if $x\in \Ex(F)$ and $x = (1 - \theta)y + \theta z$ with $y, z\in C$ and $0 < \theta < 1,$ then the fact that $F$ is a face of $C$ implies that $y, z\in F$ and thus, since $x\in\Ex(F),$ this means $y = z$. Hence, $x\in\Ex(C)$ with $x\in F,$ i.e. $x \in \Ex(C) \cap F.$}

Concerning the space $\calP_\fp(B_\rmw),$ we also need the characterization of the extreme points of a subset of a convex set with linear constraints. For that, we use the following result by Dubins on general convex sets under linear restrictions (see \cite{Dubins1962} for the proof). A related result for measures appears in \cite[Theorem 2.1]{Winkler1988}, but some of the hypotheses are more restrictive than those of Dubins and are not fulfilled in our setting.

\begin{theorem}[Dubins]
    \label{thmdubins}
    Let $A$ be a convex subset of a real vector space $X.$ Suppose $A$ is linearly bounded and linearly closed, in the sense that every line in $X$ intersects $A$ in a bounded and closed subset of the line. Let $N$ be the intersection of $A$ with $n$ hyperplanes in $X.$ Then, every extreme point in $N$ is a convex combination of at most $n+1$ extreme points of $A.$
\end{theorem}

Notice there is no topological structure in the vector space $X$ of the theorem, so the hyperplanes are simply the sets of the form $\{x\in X; \; L(x) = c\}$, for some constant $c$ and some linear function $L,$ with no continuity assumption. 

Our restrictions, however, are of the form $\{x\in X; \; L(x) \leq c\}.$ This can be broken down into two cases, whether the extremum point is in the half-space $\{L < c\}$ or in the hyperplane $\{L = c\}.$ For that, we have the following result.

\begin{proposition}
    \label{propextremewithrestriction}
    Let $A$ be a linearly bounded and linearly closed convex subset of a real vector space $X.$ Let $L:X \rightarrow \RR$ be a linear functional in $X$ and let $c\in\RR.$ Then, the extreme points in $N = A \cap \{ L \leq c\}$ are either the extreme points of $A$ or the extreme points of $N' = A \cap \{L = c\}.$ Therefore, the extreme points of $N$ are convex combinations of at most $2$ extreme points of $A.$ Moreover, in the case that an extreme point $x$ of $N'$ is a convex combination $x = (1 - \theta)y + \theta z$ of two distinct points $y, z \in A,$ with $0 < \theta < 1,$ then $y$ and $z$ are in distinct strict half-spaces divided by the hyperplane, say $L(y) < c < L(z).$
\end{proposition}

\begin{proof}
    Let $x$ be an extreme point of $N$. If $x$ belongs to $N'$, then it must also be an extreme point of $N'$ since $N'\subset N.$ In this case, \cref{thmdubins} yields the desired characterization of the extreme point.
    
    Suppose now that the extreme point $x$ of $N$ belongs to $N \setminus N'$, i.e. $L(x) < c.$ We want to argue that $x$ is an extreme point of $A.$ If $x$ were not an extreme point of $A$, then there would be two distinct points $y, z\in A$ and a real number $0 < \theta < 1$ such that $x = (1-\theta)y + \theta z.$ By linearity
    \[ L(x) = (1-\theta)L(y) + \theta L(z) < c.
    \]
    By the continuity of the real valued linear function $s \mapsto (1-s)L(y) + s L(z)$ and the fact that $0 < \theta < 1,$ there exists $\varepsilon > 0$ such that $0 < \theta - \varepsilon < \theta + \varepsilon < 1$, with
    \[ L(x_\pm) = (1 - \theta_{\pm})L(y) + \theta_\pm L(z) < c,
    \]
    for $\theta_\pm = \theta \pm \varepsilon$ and $x_\pm = (1 - \theta_\pm)y + \theta_\pm z.$ By the convexity of $A$, both $x_\pm$ belong to $A$ and, thanks to $L(x_\pm) < c,$ they belong to $N$. Since $y \neq z,$ we have $x_+ \neq x_-.$ Moreover,
    \[ x = \frac{1}{2}x_- + \frac{1}{2}x_+.
    \]
    This shows that $x$ is not an extreme point of $N,$ which is a contradiction. Thus, $x$ must also be an extreme point of $A.$

    For the last result, if $y$ and $z$ were not in distinct strict half-spaces divided by the hyperplane, then we would have say $L(y), L(z) \leq c$. But the condition $L(x) = c$ implies $L(y) = L(z) = c.$ This, however, would mean that $x$ is not an extreme point of $N',$ which would be a contradiction. Therefore, they must be in distinct strict half-spaces divided by the hyperplane.
\end{proof}

\subsection{Extreme points in measure spaces}

If $Y$ is a separable and metrizable topological space, it follows from \cite[Theorem 15.9]{AliBor2006} that the set of extreme points of $\calP(Y)$ is the set of Dirac probability measures $\delta_u$, with $u\in Y$. Since any compact and metrizable space is separable, this holds, in particular, for a compact and metrizable space $K$. We use this, together with the results of the previous section, to give, in \cref{secextremeconstrainedmeasures}, the complete characterization of the extreme points of $\calP_\fp(B_\rmw)$.

\subsection{Extreme points and Bauer's maximum principle}

The following maximum principle characterizes the maximum points of an upper semi-continuous convex function as extreme points, under compactness and convexity conditions on the domain of the function. 
\begin{theorem}[Bauer's Maximum Principle]
    \label{thmbauer}
    If $K$ is a compact convex subset of a locally convex Hausdorff space, then every upper semi-continuous convex function on $K$ has a maximum that is an extreme point.
\end{theorem}

A proof can be found in \cite[7.69 Bauer Maximum Principle]{AliBor2006}.

\section{Preliminaries}
\label{secprelim}

Here we establish the necessary results for assembling the proof of the optimal minimax formula.

\subsection{Regularity results for the spaces of measures}

The mean energy dissipation inequalities for the Foias-Prodi stationary statistical solutions imply a uniform bound on the enstrophy.
\begin{lemma}
    \label{lemmuenstrophybound}
    Let $\mu\in \calP_\rmd(H)$. Then
    \begin{equation}
        \label{eqmuenstrophybound}
        \int_H \|\bfnabla \bfu\|_{L^2}^2\;\rmd\mu(\bfu) \leq \frac{\|\bff\|_{V'}^2}{\nu^2}.
    \end{equation}
    This means that $\calP_\rmd(H) \subset \calP_\ens(H),$ and, consequently, that $\calP_\rmd(H_\rmw) \subset \calP_\ens(H_\rmw),$ $\calP_\rmd(B) \subset \calP_\ens(B),$ and $\calP_\rmd(B_\rmw) \subset \calP_\ens(B_\rmw),$ for any Borel subset $B\subset H.$
\end{lemma}

\begin{proof}
    Since $\mu\in\calP_\rmd(H),$ we have
    \[
        \int_H \left\{ \nu \|\bfnabla \bfu\|_{L^2}^2 - \dual{\bff, \bfu}_{V', V} \right\} \;\rmd\mu(\bfu) \leq 0.
    \]
    Using that
    \[
        \dual{\bff, \bfu}_{V', V} \leq \|\bff\|_{V'}\|\bfnabla \bfu\|_{L^2} \leq \frac{1}{2\nu}\|\bff\|_{V'}^2 + \frac{\nu}{2}\|\bfnabla\bfu\|_{L^2}^2,
    \]
    we obtain
    \[
        \nu \|\bfnabla \bfu\|_{L^2}^2 - \dual{\bff, \bfu}_{V', V} \geq \frac{\nu}{2} \|\bfnabla \bfu\|_{L^2}^2 - \frac{1}{2\nu}\|\bff\|_{V'}^2,
    \]
    so that
    \begin{multline*}
        \frac{\nu}{2} \int_H  \|\bfnabla \bfu\|_{L^2}^2 \;\rmd\mu(\bfu) \leq \int_H \left\{ \nu \|\bfnabla \bfu\|_{L^2}^2 - \dual{\bff, \bfu}_{V', V} + \frac{1}{2\nu}\|\bff\|_{V'}^2 \right\} \;\rmd\mu(\bfu) \\ \leq \frac{1}{2\nu} \int_H \|\bff\|_{V'}^2 \;\rmd\mu(\bfu) = \frac{1}{2\nu} \|\bff\|_{V'}^2,
    \end{multline*}
    proving the result.
\end{proof}

The uniform bound on the mean enstrophy valid for the measures in $\calP_\ens(H)$, and, in particular, for the measures in $\calP_\fp(H)$ and $\calP_\rmd(H)$, implies a (uniform) tightness property with respect to bounded sets in $V$.
\begin{lemma}
    \label{lemPensuniftightinV}
    For any $R > 0$ and any $\mu \in \calP_\ens(H)$,
    \begin{equation}\label{ineq:mu:BVRc}
        \mu\left(H \setminus \left\{\bfu\in V; \;\|\bfnabla \bfu\|_{L^2} \leq  R\right\} \right) \leq \frac{\|\bff\|_{V'}^2}{\nu^2 R^2}.
    \end{equation}
\end{lemma}

\begin{proof}
    Just use the Chebyshev inequality together with the mean enstrophy bound in \cref{meanenstrophybound}.
\end{proof}

A fundamental consequence of the uniform tightness of measures given in \cref{lemPensuniftightinV} is that the topologies in the measure spaces taken with respect to either the weak and or the strong topologies on a bounded set in $H$ coincide.

\begin{lemma}
    \label{lemPensBequalsPensBw}
    Let $B$ be a closed subset of $H$. Then,
    $\calP_\ens(B) = \calP_\ens(B_\rmw)$ as topological spaces. In other words, in these spaces, $\mu_n \wconv \mu$ means $\int_H \varphi(u) \;\rmd\mu_n(u) \rightarrow \int_H \varphi(u) \;\rmd\mu(u),$ for all $\varphi$ in either $\calC_\rmb(B)$ or $\calC_\rmb(B_\rmw).$
\end{lemma}

\begin{proof}
    Since $\calC_\rmb(B_\rmw) \subset \calC_\rmb(B),$ we immediately have the continuous inclusion $\calP_\ens(B) \subset \calP_\ens(B_\rmw).$ Now, suppose $\mu_n \wconv \mu$ in $\calP_\ens(B_\rmw).$ From \cref{lemPensuniftightinV} and the fact that, for each $R > 0$, the ball $B_V(R)$ is a compact subset of $H$, we see that $\{\mu_n\}_n$ is uniformly tight in $\calP(B).$ Since $B$ is a closed subset of $H$, then it is a Polish space, and it follows by Prokhorov's Theorem that $\{\mu_n\}_n$ has a convergent subsequence $\{\mu_{n_k}\}_k$, namely $\mu_{n_k} \wconv \nu$ in $\calP(B),$ for some $\nu\in\calP(B).$ Since $\calP(B)\subset \calP(B_\rmw)$ as well, we have $\mu_{n_k} \wconv \nu$ in $\calP(B_\rmw).$ From the fact that $\calP(B_\rmw)$ is a Hausdorff space, the limit is unique and we have $\nu = \mu.$ Thus, $\mu_{n_k} \wconv \mu$ in both $\calP(B_\rmw)$ and $\calP(B).$ Since this must happen to any convergent subsequence, then in fact the whole sequence converges to $\mu$ in $\calP(B).$ In other words, $\calP_\ens(B_\rmw) \subset \calP_\ens(B).$ This completes the proof. 
\end{proof}

The mean energy dissipation inequalities \eqref{meandissipationineqs} also imply that these measures have bounded support in $H$:

\begin{lemma}
    \label{lemmusupportedinH}
    Let $\mu\in \calP_\fp(H)$. Then $\mu$ is carried by the ball $B_H(\radius_0)$ in $H$, i.e.
    \begin{equation}
        \label{eqmuenergybound}
        \mu(H \setminus B_H(\radius_0)) = 0,
    \end{equation}
    where $\radius_0$ is given by \eqref{defrho0}.
\end{lemma}

\begin{proof}
    Consider the nonnegative, non-decreasing, smooth function with bounded derivative given by
    \[
        \psi(s) = \begin{cases}
            0, & s \leq \radius_0^2, \\
            e^{-1/(s - \radius_0^2)}, & s > \radius_0^2.
        \end{cases}
    \]
From the mean energy dissipation inequality and using that the derivative $\psi'(s)$ vanishes for $s \leq \radius_0^2$, we obtain
    \begin{multline}
        \label{eqformusupportedinH}
        \int_{H \setminus B_H(\radius_0)} \psi'(\|\bfu\|_{L^2}^2) \left\{ \nu \|\bfnabla\bfu\|_{L^2}^2 - \dual{\bff, \bfu}_{V', V} \right\} \;\rmd\mu(\bfu) \\
        = \int_H \psi'(\|\bfu\|_{L^2}^2) \left\{ \nu \|\bfnabla\bfu\|_{L^2}^2 - \dual{\bff, \bfu}_{V', V} \right\} \;\rmd\mu(\bfu) \leq 0.
    \end{multline}
    On the other hand, $\psi'(\| \cdot \|_{L^2}^2)$ is strictly positive on the complement $H \setminus B_H(\radius_0)$ and the second factor from the integrand in \eqref{eqformusupportedinH} is strictly positive on $V \setminus B_H(\radius_0)$. Indeed, using the duality norm, the Poincar\'e inequality and the definition of $\radius_0,$
    \begin{align*} 
        \nu \|\bfnabla\bfu\|_{L^2}^2 - \dual{\bff, \bfu}_{V', V} & \geq \nu \|\bfnabla\bfu\|_{L^2}^2 - \|\bff\|_{V'} \|\bfnabla\bfu\|_{L^2} \\
        & = \nu \|\bfnabla\bfu\|_{L^2} \left( \|\bfnabla\bfu\|_{L^2} - \frac{\|\bff\|_{V'}}{\nu} \right) \\
        & \geq \nu \|\bfnabla\bfu\|_{L^2} \left( \lambda_1^{1/2}\|\bfu\|_{L^2} - \frac{\|\bff\|_{V'}}{\nu} \right) \\
        & > \nu \|\bfnabla\bfu\|_{L^2} \left( \lambda_1^{1/2}\radius_0 - \frac{\|\bff\|_{V'}}{\nu} \right) \\
        & = 0,
    \end{align*}
    for all $\bfu \in V \setminus B_H(\radius_0).$ Since the measure $\mu$ is carried by $V$, we have the integrand in the left hand side of \eqref{eqformusupportedinH} strictly positive, $\mu$-almost surely on the complement $H \setminus B_H(\radius_0),$ with the corresponding integral being nonpositive. This implies that the $\mu$-measure of this complement $H \setminus B_H(\radius_0)$ must be zero, which means that $\mu$ is carried by the ball $B_H(\radius_0),$ completing the proof.
\end{proof} 

\subsection{Continuity results}

We start with continuity results for functions that appear as integrands within the proof of the minimax formula.
\begin{lemma}
    \label{lemFDPsiVweak}
    For any $\Psi\in\calT^\cyl,$ the functional
    \begin{equation}
        \label{eqdefgFDPsiVweak}
        g(\bfu) = \dual{\bfF(\bfu), \bfD\Psi(\bfu)}_{V', V}
    \end{equation}
    is well defined for $\bfu\in V$ and is continuous with respect to the weak topology of $V$. Moreover, it satisfies the bound
    \begin{equation}
        \label{eqboundonFDPsi}
        |g(\bfu)| \leq C_\Psi \left( \|\bff\|_{V'} + \nu \|\bfnabla \bfu\|_{L^2} + c_L\|\bfu\|_{L^2}^{1/2} \|\bfnabla\bfu\|_{L^2}^{3/2}\right),
    \end{equation}
    for all $\bfu\in V,$ for a constant $C_\Psi$ depending on $\Psi.$
\end{lemma}

\begin{proof}
    Using \cref{eqfrechetPsi} and the definition of $\bfF$, we have
    \[
        g(\bfu) = \sum_{j=1}^m \partial_j \psi(\dual{\bfu,\bfw_1}_{V',V}, \ldots, \dual{\bfu,\bfw_m}_{V', V}) \dual{\bff - \nu A\bfu - B(\bfu, \bfu) , \bfw_j}_{V', V}.
    \]
    If $\bfu_n$ converges weakly to $\bfu$ in $V,$ then it also converges strongly in $H,$ and we see, from this expression, that $g(\bfu_n)\rightarrow g(\bfu).$

    For the estimate, since the partial derivatives $\partial_j \psi$ of $\psi$ are uniformly bounded, we have
    \[
        |g(\bfu)| \leq \sum_{j=1}^m C_{\psi, j}\left( \|\bff\|_{V'} + \nu \|\bfnabla \bfu\|_{L^2} + \|B(\bfu, \bfu)\|_{V'} \right) \|\bfnabla \bfw_j\|_{L^2},
    \]
    where $C_{\psi, j} = \sup_{s\in\RR}|\partial_j \psi(s)|$. Thus, using \cref{eqboundBuuinVprime}, we obtain \cref{eqboundonFDPsi} with
    \[
        C_\Psi = \sum_{j=1}^m C_{\psi, j} \|\bfnabla \bfw_j\|_{L^2}.
    \]
\end{proof}

Since any sequence converging weakly in $H$ and bounded in $V$ also converges weakly in $V,$ we have the following corollary.

\begin{corollary}
    When restricted to a bounded ball in $V,$ the functional $g=g(\bfu)$ defined in \cref{eqdefgFDPsiVweak} is continuous with respect to the weak topology of $H.$ 
\end{corollary}





We now address the continuity of some integral functionals with respect to the measure.

\begin{lemma}
    \label{lemgvwGcontinuous}
    Let $B$ be a weakly compact subset of $H$. Suppose $g:V\cap B\rightarrow \RR$ is continuous in the weak topology of $V$ and is sub-quadratic with respect to the norm in $V$, in the sense that 
    \begin{equation}
        \label{subquadraticgrowth}
        |g(\bfu)| \leq C_0 + C_1\|\bfnabla \bfu\|_{L^2}^{2 - \delta}, \qquad \forall\,\bfu\in V \cap B, 
    \end{equation}
    for some $C_0, C_1 \geq 0$ and some $0 < \delta \leq 2.$ Then, the map
    \begin{equation}
        G(\mu) = \int_B g(\bfu) \;\rmd \mu(\bfu)
    \end{equation}
    is continuous for $\mu$ in $\calP_\ens(B),$ i.e. if $\mu_n\in \calP_\ens(B_\rmw)$ is such that $\mu_n \wconv \mu$ for some $\mu$ in $\calP_\ens(B_\rmw),$ then $G(\mu_n) \rightarrow G(\mu).$
\end{lemma}

\begin{proof} 
Since $g:V\cap B\rightarrow\RR$ is continuous with respect to the weak topology of $V$ and $V$ is a Borel subset of $H$, the functional $g$ is a Borel map on $V\cap B$ with respect to the weak topology of $H$. Since any measure in $\calP_\ens(B_\rmw)$ is carried by $V,$ then $g$ is $\mu$-measurable, for any $\mu$ in $\calP_\ens(B_\rmw)$. Since any measure in $\calP_\ens(B_\rmw)$ has finite mean enstrophy \cref{meanenstrophybound}, the bound on $g$ implies that $g$ is also $\mu$-integrable, for any $\mu$ in $\calP_\ens(B_\rmw)$.

 Consider now a sequence $\mu_n\wconv \mu$ in $\calP_\ens(B_\rmw)$. Let $R > 0$ be arbitrary and denote by $g_R = g|_{B_V(R)\cap B}$ the restriction of $g$ to $B_V(R)\cap B$. Observe that, since $B_V(R)\cap B$ is compact in $B_\rmw$ and $g$ is continuous on $B_V(R)\cap B$ with respect to the weak topology of $V$, it follows that $g_R$ is continuous also in the weak topology of $H.$\footnote{Otherwise, if $\bfu_n\rightarrow \bfu$ converges in $B_\rmw$ but $g(\bfu_n)$ does not converge to $g(\bfu)$, we could find a subsequence $\bfu_{n_k}$ for which $g(\bfu_{n_k})$ is away from $g(\bfu)$, but then the compactness of $B_V(R)\cap B$ in the weak topology of $V$ would imply that a further subsequence $\bfu_{n_{k_j}}$ would converge to $\bfu$ in the weak topology of $V$, meaning, by the continuity of $g$ in this topology, that $g(\bfu_{n_{k_j}})$ would converge to $g(\bfu)$.}  Since $B_V(R)\cap B$ is compact also in the weak topology of $H$ and that $B_\rmw$ is normal, we apply the Tietze Extension Theorem \cite{DunSch1958,Munkres2000} to find a bounded, weakly continuous function $h_R$ on $B$ which extends $g_R$ from $B_V(R)\cap B$ to $B.$ Moreover, the extension $h_R$ can be chosen such that
    \begin{multline*}
        \max_{\bfu\in B} |h_R(\bfu)| = \max_{\bfu\in B_V(R) \cap B} |g_R(\bfu)| \\
        \leq \sup_{\bfu\in B_V(R) \cap B} \{C_0 + C_1 \|\bfnabla \bfu\|_{L^2}^{2 - \delta}\} = C_0 + C_1 R^{2 - \delta}. 
    \end{multline*}
    Then, we write
    \[  G(\mu_n) = \int_B g(\bfu) \;\rmd \mu_n(\bfu) = \int_B (g(\bfu) - h_R(\bfu))\;\rmd \mu_n(\bfu) + \int_B h_R(\bfu) \;\rmd \mu_n(\bfu).
    \]
    Since $h_R$ is an extension of $g_R$, the difference $g - h_R$ vanishes on $B_V(R)\cap B$, thus we have, using the boundedness of  $h_R$ and the sub-quadratic property of $g$ that
    \begin{align*}
        \bigg|\int_B (g(\bfu) - & h_R(\bfu)) \;\rmd \mu_n(\bfu)\bigg| = \left|\int_{B \setminus B_V(R)}  (g(\bfu) - h_R(\bfu))\;\rmd \mu_n(\bfu)\right| \\
          & \leq \int_{B \setminus B_V(R)}  \left\{ 2C_0 + C_1 \|\bfnabla\bfu\|_{L^2}^{2 - \delta} + C_1 R^{2 - \delta} \right\}\;\rmd \mu_n(\bfu) \\
          & \leq (2C_0 + C_1 R^{2-\delta})\mu_n(B \setminus B_V(R)) + C_1 \int_{B \setminus B_V(R)} \|\bfnabla\bfu\|_{L^2}^{2 - \delta} \;\rmd\mu_n(\bfu) \\
          & \leq \left(\frac{2C_0 + C_1R^{2-\delta}}{R^2} + \frac{C_1}{R^\delta}\right)\int_B \|\bfnabla \bfu\|_{L^2}^2\;\rmd\mu_n(\bfu),
    \end{align*}
  where in the last inequality we estimated $\mu_n(B \setminus B_V(R))$ using \eqref{ineq:mu:BVRc}.
   
    Using \eqref{meanenstrophybound}, we arrive at
    \[
        \bigg|\int_B (g(\bfu) - h_R(\bfu)) \;\rmd \mu_n(\bfu)\bigg| \leq 2\left(\frac{C_0}{R^2} + \frac{C_1}{R^\delta}\right)\frac{\|\bff\|_{V'}^2}{\nu^2}.
    \]
    This holds for every $n\in\NN.$ Similarly, the same estimate holds for the limit measure $\mu\in \calP_\ens(B_\rmw).$ With that, we have
    \begin{align*}
        \left| G(\mu_n) - G(\mu) \right| & = \bigg| \int_B (g(\bfu) - h_R(\bfu))\;\rmd \mu_n(\bfu) \\
        & \qquad + \int_B h_R(\bfu) \;\rmd \mu_n(\bfu) - \int_B h_R(\bfu) \;\rmd \mu(\bfu) \\
        & \qquad + \int_B (h_R(\bfu) - g(\bfu))\;\rmd \mu(\bfu) \bigg| \\
        & \leq  4\left(\frac{C_0}{R^2} + \frac{C_1}{R^\delta}\right)\frac{\|\bff\|_{V'}^2}{\nu^2} \\
        & \qquad + \left|\int_B h_R(\bfu) \;\rmd \mu_n(\bfu) - \int_B h_R(\bfu) \;\rmd \mu(\bfu)\right|.
    \end{align*}
    Since the extension $h_R$ is bounded and weakly continuous in $H$, the convergence $\mu_n \wconv \mu$ in $\calP_\ens(B_\rmw)$ implies that
    \[
        \int_B h_R(\bfu)\;\rmd\mu_n(\bfu) \rightarrow \int_B h_R(\bfu)\;\rmd\mu(\bfu), \qquad n\rightarrow \infty.
    \]
    Thus, we have
    \[
        \limsup_{n\rightarrow \infty} \left| G(\mu_n) - G(\mu) \right| \leq  4\left(\frac{C_0}{R^2} + \frac{C_1}{R^\delta}\right)\frac{\|\bff\|_{V'}^2}{\nu^2}.
    \]
    Since $R > 0$ is arbitrary, the right hand side is arbitrarily small, and we find that
    \[
        \limsup_{n\rightarrow \infty} \left| G(\mu_n) - G(\mu) \right| = 0,
    \]
    meaning that $G(\mu_n)$ converges to $G(\mu),$ proving the continuity of $G$ on the space $\calP_\ens(B_\rmw)$.
\end{proof}

\begin{corollary}
    \label{corgvwGcontinuous}
    Let $B$ be a weakly compact subset of $H$. For any given $\phi: V\cap B\rightarrow \RR$ which is weakly-continuous in $V$ and has sub-quadratic growth on $V$ (as in \cref{lemgvwGcontinuous}), the map
    \begin{equation}
        G(\mu, \Psi) = \int_B \{\phi(\bfu) + \dual{\bfF(\bfu), \bfD\Psi(\bfu)}_{V', V}\} \;\rmd\mu(\bfu)
    \end{equation}
    is continuous in $\mu\in \calP_\ens(B_\rmw)$, for every fixed $\Psi\in\calT^\cyl.$
\end{corollary}

\begin{proof}
    Just consider $g(\bfu) = \phi(\bfu) + \dual{\bfF(\bfu), \bfD\Psi(\bfu)}_{V', V}$ and apply \cref{lemgvwGcontinuous}, using the fact that $\bfu \mapsto \dual{\bfF(\bfu), \bfD\Psi(\bfu)}_{V', V}$ is also weakly continuous in $V$ with sub-quadratic growth on $V$, which follows from \cref{lemFDPsiVweak} and the fact that $B$ is bounded in $H$.
\end{proof}


We also need continuity with respect to the test functional.

\begin{lemma}
    \label{lemgvwGcontinuousinPsi}
    For any fixed $\mu\in \calP_\ens(H_\rmw)$ and any $\phi: H\rightarrow \RR$ which is $\mu$-integrable, the map
    \begin{equation}
        G(\mu, \Psi) = \int_H \{\phi(\bfu) + \dual{\bfF(\bfu), \bfD\Psi(\bfu)}_{V', V}\} \;\rmd\mu(\bfu)
    \end{equation}
    is continuous in $\Psi\in\calT^\cyl.$ 
\end{lemma}

\begin{proof}
Recall $\calT^\cyl$ is a vector space endowed with the norm \eqref{eqcalTcylmetric}. Let $\Psi_1, \Psi_2 \in \calT^\cyl$. Since $\mu$ is carried by $V$, the function $\bfu\mapsto \dual{\bfF(\bfu), \bfD\Psi_1(\bfu)-\bfD\Psi_2(\bfu)}_{V', V}$ is also $\mu$-integrable and we have
\[
    |G(\mu, \Psi_1)-G(\mu, \Psi_2)| \leq \int_V | \dual{\bfF(\bfu), \bfD\Psi_1(\bfu)-\bfD\Psi_2(\bfu)}_{V', V} | \;\rmd\mu(\bfu).
\]
Then,
\begin{align*}
    |G(\mu, \Psi_1)-& G(\mu, \Psi_2)| \leq \int_V \|\bfF(\bfu)\|_{V'} \|\bfD\Psi_1(\bfu)-\bfD\Psi_2(\bfu)\|_V \;\rmd\mu(\bfu) \\
    & \leq \sup_{\bfv \in V} \|\bfD\Psi_1(\bfv)-\bfD\Psi_2(\bfv)\|_V \int_V \|\bfF(\bfu)\|_{V'} \;\rmd\mu(\bfu) \\
    & \leq \|\Psi_1-\Psi_2\|_{\calT^\cyl} \int_V\{\|\bff\|_{V'}+\nu\|\bfnabla \bfu\|_{L^2}+c_L\|\bfu\|_{L^2}^{1/2}\|\bfnabla\bfu\|_{L^2}^{3/2}\}\;\rmd\mu(\bfu).
\end{align*}
Then, using the Poincar\'e inequality and the bound \eqref{meanenstrophybound}, we find that
\begin{multline*}
    |G(\mu, \Psi_1) - G(\mu, \Psi_2)| \\ 
    \leq  \|\Psi_1-\Psi_2\|_{\calT^\cyl} \int_V\{\|\bff\|_{V'}+\nu\|\bfnabla \bfu\|_{L^2}+\frac{c_L}{\lambda_1^{1/4}} \|\bfnabla\bfu\|_{L^2}^2\}\;\rmd\mu(\bfu) \\ \leq C\|\Psi_1-\Psi_2\|_{\calT^\cyl},
\end{multline*}
where
\[
    C = 2\|\bff\|_{V'} + \frac{c_L}{\lambda_1^{1/4}} \frac{\|\bff\|_{V'}^2}{\nu^2}.
\]
This proves the continuity of $G(\mu, \Psi)$ with respect to $\Psi.$
\end{proof}

\subsection{Topological results for the spaces of measures}

Here we prove some fundamental topological properties of the subspaces of measures defined in \cref{secmeasurespacessss}.

\begin{lemma}
    \label{lemPensHwclosedinPHw}
    The space $\calP_\ens(H_\rmw)$ is a closed and convex subset of the space $\calP(H_\rmw)$.
\end{lemma}

\begin{proof}
    Consider a sequence $\{\mu_n\}_n$ in $\calP_\ens(H_\rmw)$ converging to some measure $\mu$ in $\calP(H_\rmw)$. Since the weak-star convergence in $\calP(H_\rmw)$ is equivalent to the lower-semi-continuous weak-star convergence and, for any constant $M\in\RR,$ the map $\bfu \mapsto \min\{M, \|\bfnabla\bfu\|_{L^2}^2\}$ is a bounded lower-semi-continuous function on $H_\rmw$\footnote{Let $f(\bfu) = \min\{M, \|\bfnabla\bfu\|_{L^2}^2\},$ if $\bfu\in V$, and $f(\bfu) = M,$ if $\bfu\in H\setminus V.$ Suppose $\bfu_n$ converges to $\bfu$ weakly in $H$. If $\liminf_n f(\bfu_n) = M,$ then automatically $f(\bfu) \leq M \leq \liminf_n f(\bfu_n),$ since $f$ is bounded by $M$. If $L = \liminf_n f(\bfu_n) < M,$ then there is a subsequence $\bfu_{n_k}\in V$ with $\|\bfnabla\bfu_{n_k}\|_{L^2}^2 < M,$ for all $k$, and with $L=\lim_k f(\bfu_{n_k})$. This means there is a further subsequence $\bfu_{n_{k_j}}$ which converges weakly in $V$, which is necessarily $\bfu,$ by the uniqueness of the limit. Since $\bfu \mapsto \|\bfnabla \bfu\|_{L^2}^2$ is well-known to be weakly lower semi-continuous in $V$, we find that $f(\bfu) \leq \liminf_j \|\bfnabla\bfu_{n_{k_j}}\|_{L^2}^2 = \liminf_j f(\bfu_{n_{k_j}}) = \lim_k f(\bfu_{n_k}) = L = \liminf_n f(\bfu_n),$ proving that $f$ is weakly lower semi-continuous in $H$.}, we have that
    \begin{multline*}
        \int_H \min\{M, \|\bfnabla\bfu\|_{L^2}^2\} \;\rmd\mu(\bfu) \leq \liminf_{n\rightarrow \infty} \int_H \min\{M, \|\bfnabla\bfu\|_{L^2}^2\} \;\rmd\mu_n(\bfu) \\
        \leq \liminf_{n\rightarrow \infty} \int_H  \|\bfnabla\bfu\|_{L^2}^2 \;\rmd\mu_n(\bfu).
    \end{multline*}
    Using the mean enstrophy bound given in \eqref{meanenstrophybound} for $\mu_n\in\calP_\ens(B_\rmw)$, we obtain that
    \[
    \int_H \min\{M, \|\bfnabla\bfu\|_{L^2}^2\} \;\rmd\mu(\bfu) \leq \frac{\|\bff\|_{V'}^2}{\nu^2}.
    \]
    Since $M$ is arbitrary, we obtain, from the Monotone Convergence Theorem, that
    \[
	   \int_H \|\bfnabla\bfu\|_{L^2}^2 \;\rmd\mu(\bfu) \leq \frac{\|\bff\|_{V'}^2}{\nu^2},
    \]
    proving that $\mu$ satisfies \eqref{meanenstrophybound}. In particular, this means that $\mu\in\calP_\ens(H_\rmw).$

    The fact that $\calP_\ens(H_\rmw)$ is convex follows from the fact that $\calP(H_\rmw)$ is convex and the fact that the condition \eqref{meanenstrophybound} characterizing $\calP_\ens(H_\rmw)$ is linear in $\mu$.
\end{proof}

Since the topology in $\calP(H)$ is finer than that in $\calP(H_\rmw)$ and since $\calP_\ens(H)$ is equal to $\calP_\ens(H_\rmw)$ as a set and the latter is closed in $\calP(H_\rmw)$ as proved in the previous result, then $\calP_\ens(H)$ is a closed subset of $\calP(H)$. Our needs, however, require us to focus on the weak topology, instead.

\begin{lemma}
    \label{lemfpcompact}
    If $B$ is a weakly compact subset of $H$, then $\calP_\fp(B_\rmw)$ is a closed and convex subset of the compact space $\calP(B_\rmw)$, hence it is also compact.
\end{lemma}

\begin{proof}    
    Consider a sequence $\{\mu_n\}_n$ in $\calP_\fp(B_\rmw)$ converging to some $\mu$ in $\calP(B_\rmw)$. Since $B$ is  weakly closed in $H$, we know that $\calP(B_\rmw)$ is a closed subset of $\calP(H_\rmw).$ We have also seen in \cref{lemPensHwclosedinPHw} that $\calP_\ens(H_\rmw)$ is a closed subset of $\calP(H_\rmw).$ Thus, we find that $\calP_\ens(B_\rmw) = \calP_\ens(H_\rmw)\cap\calP(B_\rmw)$ is a closed subset of $\calP(H_\rmw)$ and, in particular, a closed subset of $\calP(B_\rmw)$. Since the sequence $\{\mu_n\}_n$ is included in $\calP_\ens(B_\rmw)$, it follows from the closedness of $\calP_\ens(B_\rmw)$ that the limit measure $\mu$ belongs to $\calP_\ens(B_\rmw)$. This proves that $\mu$ satisfies \eqref{finitemeanenstrophy}.

    For verifying \eqref{meandissipationineqs}, let $\psi$ be a real-valued continuously-differentiable function which is nonnegative, non-decreasing and with bounded derivative. We start with the first term in the integrand in \eqref{meandissipationineqs}. For that, we use \cref{lemPensBequalsPensBw} and work with the weak convergence of measures with respect to the strong topology of $H.$ Since $\psi'$ is nonnegative and $V$ is compactly immersed in $H$, the function
    \[
        \bfu \mapsto \psi'(\|\bfu\|_{L^2}^2) \min\{M, \|\bfnabla\bfu\|_{L^2}^2\}
    \]
    is bounded and strongly lower semi-continuous in $H$. Thus, we obtain, in this case, that
    \begin{multline}
        \label{eqliminfpsiprimenablau}
        \int_H \psi'(\|\bfu\|_{L^2}^2) \min\{M, \|\bfnabla\bfu\|_{L^2}^2\} \;\rmd\mu(\bfu) \\
        \leq \liminf_{n\rightarrow \infty} \int_H \psi'(\|\bfu\|_{L^2}^2) \min\{M, \|\bfnabla\bfu\|_{L^2}^2\} \;\rmd\mu_n(\bfu).
    \end{multline}
    Since $\psi'$ is nonnegative, the integrand on the right hand side of \eqref{eqliminfpsiprimenablau} is non-decreasing in $M$, so that
    \[
        \int_H \psi'(\|\bfu\|_{L^2}^2) \min\{M, \|\bfnabla\bfu\|_{L^2}^2\} \;\rmd\mu(\bfu) \leq \liminf_{n\rightarrow \infty} \int_H \psi'(\|\bfu\|_{L^2}^2) \|\bfnabla\bfu\|_{L^2}^2 \;\rmd\mu_n(\bfu).
    \]
    Taking $M\rightarrow \infty$ and applying the Monotone Convergence Theorem, we obtain
    \begin{equation}
        \label{eqliminfpsiprimenablau2}
        \int_H \psi'(\|\bfu\|_{L^2}^2) \|\bfnabla\bfu\|_{L^2}^2 \;\rmd\mu(\bfu) \leq \liminf_{n\rightarrow \infty} \int_H \psi'(\|\bfu\|_{L^2}^2) \|\bfnabla\bfu\|_{L^2}^2 \;\rmd\mu_n(\bfu).
    \end{equation}
    
    We now look at the second term in the integrand in \eqref{meandissipationineqs}. We let
    \[
    \varphi(\bfu) = \psi'(\|\bfu\|_{L^2}^2) \dual{\bff, \bfu}_{V', V}.
    \]
    This term is weakly continuous in $V$, with the bound
    \[
        |\varphi(\bfu)| \leq (\sup \psi') \|\bff\|_{V'}\|\bfnabla\bfu\|_{L^2},
    \]
    so that $\varphi$ satisfies the conditions in \cref{lemgvwGcontinuous}, with $\delta = 1.$ Applying this lemma gives that the corresponding $G(\mu_n)$ converges to $G(\mu),$ which means that
    \begin{equation}
        \label{eqlimpsiprimedual}
        \int_H \psi'(\|\bfu\|_{L^2}^2) \dual{\bff, \bfu}_{V', V} \;\rmd\mu(\bfu) = \lim_{n\rightarrow \infty} \int_H \psi'(\|\bfu\|_{L^2}^2) \dual{\bff, \bfu}_{V', V} \;\rmd\mu_n(\bfu).
    \end{equation}
    Combining \eqref{eqliminfpsiprimenablau2} with \eqref{eqlimpsiprimedual} yields that
    \begin{multline*}
        \int_H \psi'(\|\bfu\|_{L^2}^2) \left\{ \nu\|\bfnabla\bfu\|_{L^2}^2 - \dual{\bff, \bfu}_{V', V} \right\}\;\rmd\mu(\bfu) \\ 
        \leq \liminf_{n\rightarrow \infty} \int_H \psi'(\|\bfu\|_{L^2}^2) \left\{ \nu\|\bfnabla\bfu\|_{L^2}^2 - \dual{\bff, \bfu}_{V', V} \right\}\;\rmd\mu_n(\bfu).
    \end{multline*}
    Since each $\mu_n$ belongs to $\calP_\fp(B_\rmw)$, the right hand side above is nonpositive, which implies that
    \begin{equation}
        \int_H \psi'(\|\bfu\|_{L^2}^2) \left\{ \nu\|\bfnabla\bfu\|_{L^2}^2 - \dual{\bff, \bfu}_{V', V} \right\}\;\rmd\mu(\bfu) \leq 0.
    \end{equation}
    Since $\psi$ is arbitrary, we find that the limit measure $\mu$ satisfies the mean energy dissipation inequality \eqref{meandissipationineqs}. This completes the proof that $\mu\in \calP_\fp(B_\rmw)$ and, hence, that $\calP_\fp(B_\rmw)$ is closed.
    
    The fact that $\calP_\fp(B_\rmw)$ is convex follows from the facts that $\calP(B_\rmw)$ is convex and that the conditions \eqref{finitemeanenstrophy} and \eqref{meandissipationineqs} characterizing $\calP_\fp(B_\rmw)$ are linear in $\mu$.    
\end{proof}

\begin{lemma}
    \label{lemfpssscompact}
    If $B$ is a weakly compact subset of $H$, then $\calP_\fpsss(B_\rmw)$ is a closed and convex subset of the compact space $\calP_\fp(B_\rmw)$, hence it is also compact.
\end{lemma}

\begin{proof}
    Let $\{\mu_n\}_n$ be a sequence in $\calP_\fpsss(B_\rmw)$ such that $\mu_n \wconv \mu$ in $\calP(B_\rmw)$. Since $\calP_\fpsss(B_\rmw) \subset \calP_\fp(B_\rmw)$ and, from \cref{lemfpcompact}, $\calP_\fp(B_\rmw)$ is closed, then $\mu \in \calP_\fp(B_\rmw)$. Thus, it only remains to verify that $\mu$ satisfies \eqref{eqliouville}.
    
    Let $\Psi \in \calT^\cyl$ be fixed and define the functional $g(\bfu) = \dual{\bfF(\bfu), \bfD\Psi(\bfu)}_{V', V}$. 
    According to \cref{lemFDPsiVweak}, $g$ is continuous with respect to the weak topology of $V$ and satisfies the sub-quadratic growth condition \cref{subquadraticgrowth} since $B$ is bounded in $H$. Because $\mu_n \wconv \mu$ in $\calP_\fp(B_\rmw)$, it follows from \cref{lemgvwGcontinuous} that
    \[
        \int_B \dual{\bfF, \bfD\Psi}_{V', V} \;\rmd\mu_n(\bfu) \rightarrow \int_B \dual{\bfF, \bfD\Psi}_{V', V} \;\rmd\mu(\bfu) \quad \mbox{as } n \to \infty.
    \]
    Since $\mu_n \in \calP_\fpsss(B_\rmw)$ for each $n$, the left-hand side is zero for all $n$, implying the integral against $\mu$ is also zero. Since $\Psi$ was arbitrary in $\calT^\cyl$, we conclude that $\mu$ satisfies \cref{eqliouville}, and thus $\mu \in \calP_\fpsss(B_\rmw)$. Convexity is immediate from the linearity of the Liouville equation in $\mu$.
\end{proof}

\subsection{Characterization of the extreme points}
\label{secextremeconstrainedmeasures}

In this section, we give a precise characterization of the extreme points of $\calP_\fp(B_\rmw).$

\begin{lemma}
    \label{lemmuPfpcarriedonenergyshell}
    Let $B$ be a weakly compact subset of $H$. If $\mu$ is an extreme point of $\calP_\fp(B_\rmw),$ then it is carried by a single energy shell, i.e. there exists $e \geq 0$ such that $\mu(\{\bfu\in B; \|\bfu\|_{L^2}^2 = e\}) = 1.$
\end{lemma}

\begin{proof}
    First of all, recall that $\bfu \mapsto \|\bfu\|_{L^2}$ is a weakly lower-semi-continuous map in $H$, so that the energy shells are Borel subsets of $H_\rmw.$
    
    Now, suppose that $\mu$ is not carried by a single energy shell. Then, we can find\footnote{Consider the cumulative distribution function $F(e) = \mu(\left\{\bfu \in B; \; \|\bfu\|_{L^2}^2 \leq e\right\})$, which is a monotonic nondecreasing function with an at most countable number of discontinuities. If the corresponding distribution is not a Dirac delta measure, we can find at least two points of continuity of $F$ with $0 \leq F(e_1) < F(e_2) < 1.$ Since they are continuity points, the corresponding energy shells have zero measure, while the energy annulus has measure $0 < F(e_2) - F(e_1) < 1.$} $0 < e_1 < e_2 < \infty$ such that the associated energy shells have zero measure,
    \[ \mu(\left\{\bfu \in B; \; \|\bfu\|_{L^2}^2 = e_1\right\}) = \mu(\left\{\bfu \in B; \; \|\bfu\|_{L^2}^2 = e_2\right\}) = 0,
    \]
    and the corresponding energy annulus
    \[ 
        A_e = \left\{\bfu \in B; \; e_1 < \|\bfu\|_{L^2}^2 < e_2 \right\}
    \]
    has partial measure
    \[
        0 < \mu(A_e) < 1.
    \]

    Decompose the measure $\mu$ according to
    \[
        \mu = (1 - \theta)\mu_1 + \theta\mu_2, 
    \]
    where
    \[
        \theta = \mu(B\setminus A_e) = 1 - \mu(A_e)
    \]
    and
    \[
        \mu_1(E) =  \frac{1}{\mu(A_e)}\mu(E \cap A_e), \quad \mu_2(E) = \frac{1}{\mu(B\setminus A_e)}\mu(E \setminus A_e),
    \]
    for any Borel set $E\subset B_\rmw.$ Let us show that $\mu_1$ and $\mu_2$ belong to $\calP_\fp(B_\rmw).$

    Condition \eqref{finitemeanenstrophy}, for $\mu_1$ and $\mu_2,$ immediately follows from the linear relation
    \[
        (1 - \theta)\int_H \|\bfnabla \bfu\|_{L^2}^2 \;\rmd\mu_1(\bfu) + \theta \int_H \|\bfnabla \bfu\|_{L^2}^2 \;\rmd\mu_2(\bfu) = \int_H \|\bfnabla \bfu\|_{L^2}^2 \;\rmd\mu(\bfu) < \infty.
    \]

    For the condition \eqref{meandissipationineqs}, let $\psi$ be a real-valued continuously-differentiable function on $[0, \infty)$ which is nonnegative, non-decreasing and with bounded derivative.
    Let $\xi$ be a continuous, nonnegative real-valued function supported on the interval $(e_1, e_2).$ Consider the function $\tilde\psi(e) = \int_0^e \xi(s)\psi'(s)\;\rmd s,$ which is a nonnegative, non-decreasing continuously-differentiable function on $[0, \infty)$ such that $\tilde\psi' = \xi \psi'.$ Notice $\tilde\psi$ satisfies the conditions in \eqref{meandissipationineqs}. Then, since $\tilde\psi'$ is supported on $(e_1, e_2)$ and $\mu_2$ is carried by $B \setminus A_e,$ we have that
    \begin{align*}
        (1 - \theta) \int_H \tilde\psi'(\|\bfu\|_{L^2}^2) &  \left\{\nu\|\bfnabla\bfu\|_{L^2}^2 - \dual{\bff, \bfu}_{V', V} \right\}\;\rmd\mu_1(\bfu) \\ 
        & = (1 - \theta)\int_H \tilde\psi'(\|\bfu\|_{L^2}^2) \left\{\nu\|\bfnabla\bfu\|_{L^2}^2 - \dual{\bff, \bfu}_{V', V} \right\}\;\rmd\mu_1(\bfu) \\
        & \qquad + \theta \int_H \tilde\psi'(\|\bfu\|_{L^2}^2) \left\{\nu\|\bfnabla\bfu\|_{L^2}^2 - \dual{\bff, \bfu}_{V', V} \right\}\;\rmd\mu_2(\bfu) \\
        & = \int_H \tilde\psi'(\|\bfu\|_{L^2}^2) \left\{\nu\|\bfnabla\bfu\|_{L^2}^2 - \dual{\bff, \bfu}_{V', V} \right\}\;\rmd\mu(\bfu).
    \end{align*}
    Using condition \eqref{meandissipationineqs} for $\mu,$ we find that 
    \[
        \int_H \xi(\|\bfu\|_{L^2}^2)\psi'(\|\bfu\|_{L^2}^2) \left\{\nu\|\bfnabla\bfu\|_{L^2}^2 - \dual{\bff, \bfu}_{V', V} \right\}\;\rmd\mu_1(\bfu) \leq 0.
    \]
    Since $\xi$ is an arbitrary nonnegative, continuous function with support on the open interval $(e_1, e_2),$ we can take a sequence $\xi_n$ converging monotonically to the characteristic function of the interval and deduce, with the Lebesgue Dominated Convergence Theorem, that
    \[ \int_H \psi'(\|\bfu\|_{L^2}^2) \left\{\nu\|\bfnabla\bfu\|_{L^2}^2 - \dual{\bff, \bfu}_{V', V} \right\}\;\rmd\mu_1(\bfu) \leq 0.
    \]
    Since $\psi$ is arbitrary, we see that $\mu_1$ satisfies \eqref{meandissipationineqs}.

    The argument is similar for $\mu_2,$ choosing $\xi$ supported on the (relatively open) set $[0, e_1) \cup (e_2, \infty).$ Notice the condition that the energy levels $e_1$ and $e_2$ have zero measure allow us to apply the Lebesgue Dominated Convergence Theorem in this case as well, to obtain, at the limit of $\xi_n$ converging to the characteristic function of $[0, e_1) \cup (e_2, \infty),$ that 
    \[ \int_H \psi'(\|\bfu\|_{L^2}^2) \left\{\nu\|\bfnabla\bfu\|_{L^2}^2 - \dual{\bff, \bfu}_{V', V} \right\}\;\rmd\mu_2(\bfu) \leq 0,
    \]
    proving that $\mu_2$ also satisfies \eqref{meandissipationineqs}.
    
    Since $\mu_1$ and $\mu_2$ are distinct measures in $\calP_\fp(B_\rmw),$ and $\mu$ is a non-trivial convex combination of these two measures, we have a contradiction with the hypothesis that $\mu$ is an extreme measure in $\calP_\fp(B_\rmw).$ Therefore, any extreme measure in this space must be carried by a single energy shell.
\end{proof}

\begin{lemma}
    \label{lemextremePensd}
    Let $B$ be a weakly compact subset of $H$. A measure $\mu\in\calP(B_\rmw)$ is an extreme point of $\calP_\fp(B_\rmw)$ if, and only if, it is carried by a single energy shell and it is an extreme point of the space $\calP_\rmd(B_\rmw).$
\end{lemma}

\begin{proof}
    If $\mu$ is an extreme point of $\calP_\fp(B_\rmw),$ we see from \cref{lemmuPfpcarriedonenergyshell} that $\mu$ is carried by a single energy shell, say $S_e = \{\|\bfu\|_{L^2}^2 = e \}.$ If $\mu$ were not an extreme point of $\calP_\rmd(B_\rmw),$ we could write
    \[
        \mu = (1 - \theta) \mu_1 + \theta \mu_2,
    \]
    for distinct $\mu_1, \mu_2 \in \calP_\rmd(B_\rmw)$ and some $0 < \theta < 1.$ Both $\mu_1$ and $\mu_2$ would have to be carried by $S_e$ as well. With that, we see that $\psi'(\|\bfu\|_{L^2}^2) = \psi'(e)$ is constant almost surely with respect to both $\mu_1$ and $\mu_2$, for any $\psi$ allowed in \eqref{meandissipationineqs}. Thus,
    \begin{multline}
        \label{ineqLpsiL}
        \int_H \psi'(\|\bfu\|_{L^2}^2)\left( \nu \|\bfnabla \bfu\|_{L^2}^2 - \dual{\bff, \bfu}_{V', V} \right) \;\rmd\mu_i(\bfu) \\ = \psi'(e)\int_H \left( \nu \|\bfnabla \bfu\|_{L^2}^2 - \dual{\bff, \bfu}_{V', V} \right) \;\rmd\mu_i(\bfu) \leq 0.
    \end{multline}
    This means that $\mu_1, \mu_2\in \calP_\fp(B_\rmw).$ But this implies that $\mu$ is not an extreme point of $\calP_\fp(B_\rmw),$ which is a contradiction. Thus, $\mu$ has to be an extreme point of $\calP_\rmd(B_\rmw),$ as well.
    
    Now, let $\mu$ be an extreme point of $\calP_\rmd(B_\rmw)$ carried by a single energy shell. As done in \eqref{ineqLpsiL}, this means that $\mu$ belongs to $\calP_\fp(B_\rmw).$ 
    Then, since $\mu\in\calP_\fp(B_\rmw)$ is an extreme point of $\calP_\rmd(B_\rmw)$ and $\calP_\fp(B_\rmw) \subset \calP_\rmd(B_\rmw),$ the measure $\mu$ must also be an extreme point of $\calP_\fp(B_\rmw).$
\end{proof}

\begin{lemma}
    \label{lemPdfaceofP}
    Let $B$ be a subset of $H$. The space $\calP_\rmv(B_\rmw)$ is a face of $\calP(B_\rmw).$
\end{lemma}

\begin{proof}
    Let $H(\mu) = \int_H \|\bfnabla\bfu\|_{L^2}^2 \;\rmd\mu(\bfu),$ for any $\mu\in\calP(B_\rmw).$ Suppose $\mu_1, \mu_2 \in \calP(B_\rmw)$ and $0 < \theta < 1$ are such that $\mu = (1-\theta)\mu_1 + \theta \mu_2 \in \calP_\rmv(B_\rmw).$ By the linearity of $H,$ we find that $(1 - \theta) H(\mu_1) + \theta H(\mu_2) = H(\mu) < \infty.$ Then, since $H(\mu_1), H(\mu_2) \geq 0,$ we see that $H(\mu_1) \leq H(\mu) / (1 - \theta) < \infty$ and $H(\mu_2) \leq H(\mu) / \theta < \infty.$ This means that $\mu_1, \mu_2 \in \calP_\rmv(B_\rmw),$ proving that $\calP_\rmv(B_\rmw)$ is a face of $\calP(B_\rmw).$
\end{proof}

\begin{lemma}
    \label{lemextremePv}
    Let $B$ be a weakly compact subset of $H$. The extreme points of $\calP_\rmv(B_\rmw)$ are the Dirac delta measures $\mu = \delta_\bfu$ with $\bfu\in B \cap V.$
\end{lemma}

\begin{proof}
    Here we use that $\calP_\rmv(B_\rmw)$ is a face of $\calP(B_\rmw),$ as proved in \cref{lemPdfaceofP}. This implies that $\Ex(\calP_\rmv(B_\rmw)) = \Ex(\calP(B_\rmw)) \cap \calP_\rmv(B_\rmw),$ where we recall the notation from \eqref{def:extreme:set}. Since the extreme points of $\calP(B_\rmw)$ are the Dirac delta measures $\delta_\bfu$ with $\bfu \in B,$ the condition that $\delta_\bfu \in \calP_\rmv(B_\rmw)$ means that $\bfu \in V,$ proving the result.
\end{proof}

\begin{lemma}
    \label{lemextremePd}
    Let $B$ be a weakly compact subset of $H$. A measure $\mu$ is an extreme point of $\calP_\rmd(B_\rmw)$ if, and only if, $\mu$ is either an extreme point of $\calP_\rmv(B_\rmw)$ with $L(\mu) \leq 0$ or a convex combination $\mu = (1-\theta)\mu_1 + \theta\mu_2$ of extreme points $\mu_1, \mu_2$ of $\calP_\rmv(B_\rmw),$ with $L(\mu_1) < 0 < L(\mu_2)$ and $L(\mu) = 0,$ where $L=L(\cdot)$ is given in \eqref{defLmu}.
\end{lemma}

\begin{proof}
    Here we apply \cref{propextremewithrestriction}, stemming from Dubins' \cref{thmdubins}. We take as the ambient vector space $X$ the space of signed Radon measures with finite mean enstrophy, i.e.
    \[
        X = \left\{\nu\in \calC_\rmb(B_\rmw)'; \; \int_B \|\bfnabla \bfu\|_{L^2}^2 \;\rmd|\nu|(\bfu) < \infty \right\}.
    \]
    Recall we do not need any completeness for Dubins' Theorem; we just need the vector space structure. We take the constant $c = 0$ and the linear functional $L=L(\nu),$ for $\nu\in X,$ as in \eqref{defLmu}, which is well-defined and finite in $X,$ as well. In this framework, we see that $A = \calP_\rmv(B_\rmw) = \calP(B_\rmw) \cap X$ is a convex subset of $X$, and $N = \calP_\rmd(B_\rmw) = \{\mu \in \calP_\rmv(B_\rmw); \;L(\mu) \leq 0\} = A \cap \{L \leq 0\}.$ We just need to check that $A$ is linearly bounded and linearly closed in $X$.

    Any line in $X$ is of the form $S = \{(1 - t)\nu_1 + t\nu_2; \;t\in\RR\},$ with $\nu_1, \nu_2 \in X,$ $\nu_1\neq \nu_2.$ It intercepts $A$ at a point $\mu = (1-t)\nu_1 + t\nu_2$ for a certain $t \in \RR$ if, and only if, $\mu$ is a Borel probability measure. This means (i) $(1 - t)\nu_1(B) + t\nu_2(B) = 1$, and (ii) $(1 - t)\nu_1(E) + t \nu_2(E) \geq 0,$ for any Borel set $E\subset B.$ In order to characterize the intersection $A \cap S$, we consider the corresponding set of parameters $J = \{t\,:\, (1 - t)\nu_1 + t\nu_2 \in A \cap S \}$, and split the proof into the following cases: 
    
    {\it Case 1:} $\nu_1(B) \neq \nu_2(B)$.
    Here, there is only one value of $t\in\RR$ such that condition (i) is satisfied, i.e. $(1 - t)\nu_1(B) + t\nu_2(B) = 1$. In this case, if condition (ii) is met, $J$ is a point, otherwise, it is empty.
    
    {\it Case 2:} $\nu_1(B) = \nu_2(B) \neq 1$.
    From condition (i), it follows that in this case $J$ is the empty set.

{\it Case 3:} $\nu_1(B) = \nu_2(B) = 1$.
    Here, condition (i) is trivially satisfied. Regarding condition (ii), first note that for every Borel set $E$ such that $\nu_1(E) = \nu_2(E)$, we must have $\nu_1(E) \geq 0$ for (ii) to be satisfied, which does not entail any condition on $t$. Now, consider the Borel sets $E$ for which $\nu_1(E) \neq \nu_2(E)$. Fix one such set $E$ and assume, without loss of generality, that $\nu_2(E) > \nu_1(E)$. Then, the condition $(1 - t)\nu_1(E) + t \nu_2(E) \geq 0$ implies $t \geq -\nu_1(E)/ (\nu_2(E) - \nu_1(E))$. On the other hand, since we also have $\nu_2(E^c) < \nu_1(E^c)$, then from $(1 - t)\nu_1(E^c) + t \nu_2(E^c) \geq 0$, we obtain that $t \leq \nu_1(E^c)/(\nu_1(E^c) - \nu_2(E^c)).$ Since these inequalities must hold for all Borel sets $E$ with $\nu_1(E) \neq \nu_2(E)$, we deduce that $t$ must belong to the intersection of all the corresponding intervals, which thus characterizes the set $J$ as
    \[
    J = \bigcap_{\{E: \nu_1(E) \neq \nu_2(E)\}} \left[ \frac{\nu_1(E)}{\nu_1(E) - \nu_2(E)},  \frac{\nu_1(E^c)}{\nu_1(E^c) - \nu_2(E^c)} \right].
    \]
Since the left end point of each interval is not necessarily smaller than the right end point, the set $J$ could be empty, or a single point, or a proper closed and bounded interval.

In any case, the set $J$ is a compact set in $\mathbb{R}.$ Since $S$ is the image of $J$ under the mapping $t \in \RR \mapsto (1-t)\nu_1 + t\nu_2 \in X$ and since this mapping is continuous with respect to the weak topology in $X$, we deduce that the image of $J$ is compact in $X$. Therefore, we conclude that the intersection of $A$ with the line $\{(1 - t)\nu_1 + t\nu_2; \;t\in\RR\}$ is linearly bounded and linearly closed in $X$.

    Thus, all the conditions of \cref{propextremewithrestriction} are met and we deduce that an extreme point $\mu$ of $N = \calP_\rmd(B_\rmw)$ is either an extreme point of $A = \calP_\rmv(B_\rmw)$ with $L(\mu) \leq 0$ or a convex combination $\mu = (1 - \theta) \mu_1 + \theta \mu_2$ of two distinct extreme points $\mu_1, \mu_2$ of $A = \calP_\rmv(B_\rmw),$ with $0 < \theta < 1$, $L(\mu_1) < 0 < L(\mu_2)$ and $L(\mu) = 0.$
\end{proof}

\begin{proposition}
    Let $B$ be a weakly compact subset of $H$. If $\mu$ is an extreme point of $\calP_\fp(B_\rmw),$ then it is either a Dirac delta measure $\mu = \delta_\bfu$ with $\bfu\in B \cap V$ and $\ell(\bfu) \leq 0$ or a convex combination $\mu = (1 - \theta)\delta_{\bfu_1} + \theta \delta_{\bfu_2}$ of two distinct Dirac delta measures with $\bfu_1, \bfu_2\in B \cap V$, $\|\bfu_1\|_{L^2} = \|\bfu_2\|_{L^2},$ $\ell(\bfu_1) < 0,$ $\ell(\bfu_2) > 0$ and $0 < \theta < 1$ such that $(1-\theta) \ell(\bfu_1) + \theta \ell(\bfu_2) = 0,$ i.e. $\theta = - \ell(\bfu_1) / (\ell(\bfu_2) - \ell(\bfu_1))$, where $\ell=\ell(\cdot)$ is given in \eqref{defellmu}.
\end{proposition}

\begin{proof}
    From \cref{lemextremePensd}, we see that
    \[
        \Ex(\calP_\fp(B_\rmw)) = \Ex(\calP_\rmd(B_\rmw)) \cap \{\mu\in\calP_\rmd; \;\exists \,e \geq 0, \;\mu(S_e) = 1 \},
    \]
    where $S_e = \{\|\bfu\|_{L^2}^2 = e \}.$ Thus, we need to characterize the extreme points of $\calP_\rmd(B_\rmw)$ and select those that are carried by a single energy shell.
    
    The extreme points of $\calP_\rmd(B_\rmw)$ are characterized in \cref{lemextremePd}. Thus, the extreme points of $\calP_\fp(B_\rmw)$ are either the extreme  points $\mu$ of $\calP_\rmv(B_\rmw)$ with $L(\mu) \leq 0$ and carried on a single energy shell or a convex combination $\mu = (1-\theta)\mu_1 + \theta\mu_2$ of extreme points $\mu_1, \mu_2$ of $\calP_\rmv(B_\rmw),$ with $L(\mu_1) < 0 < L(\mu_2),$ $L(\mu) = 0,$ and carried on the same energy shell.

    According to \cref{lemextremePv}, the extreme points of $\calP_\rmv(B_\rmw)$ are the Dirac delta measures $\mu = \delta_\bfu$ with $\bfu\in B \cap V,$ which are automatically carried by a single energy shell.

    Thus, we conclude that an extreme measure $\mu$ of $\calP_\fp(B_\rmw)$ is either a Dirac delta measures $\mu = \delta_\bfu$ with $\bfu \in B \cap V$ and $L(\delta_\bfu) = \ell(\bfu) \leq 0$ or a convex combination $\mu = (1 - \theta)\delta_{\bfu_1} + \theta \delta_{\bfu_2}$ of two distinct Dirac delta measures with $\bfu_1, \bfu_2\in B \cap V$, $0 < \theta < 1,$ $L(\delta_{\bfu_1}) = \ell(\bfu_1) < 0,$ $L(\delta_{\bfu_2}) = \ell(\bfu_2) > 0$ and $L(\mu) = (1 - \theta)\ell(\bfu_1) + \theta \ell(\bfu_2) = 0,$ and carried on the same energy level, i.e. $\|\bfu_1\|_{L^2}^2 = \|\bfu_2\|_{L^2}^2.$ This concludes the proof.
\end{proof}

We can write the extreme points of $\calP_\fp(B_\rmw)$ in the following form, combining both Dirac delta measures and convex combinations of two Dirac delta measures.

\begin{corollary}
    \label{propextremePfp}
    Let $B$ be a weakly compact subset of $H$. The set of extreme points of $\calP_\fp(B_\rmw)$ can be written in the combined form
    \[
        \Ex(\calP_\fp(B_\rmw)) = \left\{ \mu\in\calP(B_\rmw); \; \mu = \theta_1 \delta_{\bfu_1} + \theta_2 \delta_{\bfu_2}, (\theta_1, \theta_2, \bfu_1, \bfu_2) \in \calB_f \right\},
    \]
    where
    \[
         \calB_f = \calB_{f, 1} \cup \calB_{f, 2},
    \]
    with the Dirac delta measures associated with
    \[
        \calB_{f, 1} = \left\{ (1, 0, \bfu, \bfu), \;\bfu \in B \cap V, \;\ell(\bfu) \leq 0
        \right\}
    \]
    and the convex combinations of two Dirac delta measures associated with
    \[ \calB_{f, 2} = \left\{ (\theta_1, \theta_2, \bfu_1, \bfu_2), \;   
        \begin{aligned}
            & 0 < \theta_1 < 1, \;\theta_2 = 1 - \theta_1, \\
            & \bfu_1, \bfu_2 \in B \cap V, \;\|\bfu_1\|_{L^2} = \|\bfu_2\|_{L^2}, \\
            & \ell(\bfu_1) < 0 < \ell(\bfu_2), \;\theta_1\ell(\bfu_1) + \theta_2 \ell(\bfu_2)  = 0
        \end{aligned}
        \right\}.
    \]
\end{corollary}

\subsection{Selection principle}

The following principle distinguishing stationary statistical solutions within the space $\calP_\fp$ is a fundamental tool in the minimax formula.
\begin{lemma}
    \label{leminfintinfty}
    Let $B$ be a weakly compact subset of $H$. In $\calP_\fp(B_\rmw)$, we have the following characterization
    \[
        \inf_{\Psi\in\calT^\cyl} \int_B \dual{\bfF, \bfD\Psi}_{V', V} \;\rmd\mu = \begin{cases}
            0, & \textrm{if } \mu \in \calP_\fpsss(B_\rmw), \\
            -\infty, & \textrm{if } \mu \in \calP_\fp(B_\rmw) \setminus \calP_\fpsss(B_\rmw).
        \end{cases}
    \]
\end{lemma}

\begin{proof}
    If $\mu\in \calP_\fp(B_\rmw)\setminus \calP_\fpsss(B_\rmw)$, it means the Liouville equation \eqref{eqliouville} is not satisfied, and there exists at least one $\Psi_0\in\calT^\cyl$ such that
    \[
        \int_H \dual{\bfF(\bfu), \bfD\Psi_0(\bfu)}_{V', V} \;\rmd\mu(\bfu) = \alpha \neq 0,
    \]
    for some $\alpha.$ Since this integral is linear in the test functional, we multiply it by $\lambda\in\RR$ to find that
    \[
        \int_H \dual{\bfF(\bfu), \bfD(\lambda\Psi_0)(\bfu)}_{V', V} \;\rmd\mu(\bfu) = \lambda\alpha.
    \]
    Since $\lambda$ is arbitrary and $\alpha \neq 0$, we find that
    \[
        \inf_{\lambda\in\RR} \int_H \dual{\bfF(\bfu), \bfD(\lambda\Psi_0)(\bfu)}_{V', V} \;\rmd\mu(\bfu) = -\infty.
    \]
    Since $\calT^\cyl$ is a vector subspace, we have $\lambda\Psi_0\in\calT^\cyl,$ for any $\lambda\in\RR$, which means that
    \[
        \inf_{\Psi\in\calT^\cyl} \int_H \dual{\bfF(\bfu), \bfD\Psi(\bfu)}_{V', V} \;\rmd\mu(\bfu) = -\infty.
    \]
    On the other hand, if $\mu\in \calP_\fpsss(B_\rmw)$, then the Liouville equation \eqref{eqliouville}, valid for any $\Psi\in\calT^\cyl,$ says that
    \[
        \inf_{\Psi\in\calT^\cyl} \int_H \dual{\bfF(\bfu), \bfD\Psi(\bfu)}_{V', V} \;\rmd\mu(\bfu) = 0.
    \]
    This completes the characterization.
\end{proof}

\section{Main result}
\label{secmainresult}

We are now ready to prove our main result.

\begin{theorem}
    \label{maintheorem}
    Assume the settings described in \cref{sec3dnse}, in particular with $\Omega\subset\RR^3$ bounded and with $\bff\in V'.$ Suppose $B$ is a weakly compact subset of $H$ for which $\calP_\fpsss(B)$ is non-empty. Then, for any $\phi: V\cap B\rightarrow \RR$ which is weakly-continuous in $V$ and has sub-quadratic growth on $V$ (as in \cref{lemgvwGcontinuous}), it follows that
    \begin{multline}
        \label{minimaxformula}
        \max_{\mu\in\calP_\fpsss(B)} \int_H \phi(\bfu) \;\rmd\mu(\bfu) \\
        = \inf_{\Psi\in\calT^\cyl} \max_{ (\bfu_1, \bfu_2, \theta_1, \theta_2)\in \calB_f} \;\sum_{i=1}^2\theta_i\{\phi(\bfu_i) + \dual{\bfF(\bfu_i), \bfD\Psi(\bfu_i)}_{V', V} \},
    \end{multline}
    where $\calB_f$ is given in \cref{propextremePfp}.
\end{theorem}

\begin{proof}
    We prove the result in a series of steps, as outlined in the Introduction.
    
    \textbf{Step 1 (change to weak topology):} We recall that the Borel sets of $H$ in the strong topology coincide with the Borel sets of $H$ in the weak topology, thus, as sets, $\calP_\fpsss(B)$ equals $\calP_\fpsss(B_\rmw)$. We also observe, from \cref{lemfpssscompact}, that $\calP_\fpsss(B_\rmw)$ is compact, while, by \cref{lemgvwGcontinuous}, the map $\mu \mapsto \int_H \phi(\bfu) \;\rmd\mu(\bfu)$ is continuous on $\calP_\fpsss(B_\rmw)$. Thus, the supremum of this map is achieved on this set, and we can indeed write the supremum as a maximum, on either set, and find that
    \[
        \max_{\mu\in\calP_\fpsss(B)} \int_H \phi(\bfu) \;\rmd\mu(\bfu) = \max_{\mu\in\calP_\fpsss(B_\rmw)} \int_H \phi(\bfu) \;\rmd\mu(\bfu).
    \]

    \textbf{Step 2 (add auxiliary functional):} By construction, the measures on $\calP_\fpsss(B_\rmw)$ satisfy the Liouville equation
    \[
        \int_H \dual{\bfF(\bfu), \bfD\Psi(\bfu)}_{V', V} \;\rmd\mu(\bfu) = 0,
    \]
    for any $\Psi\in\calT^\cyl,$ so we are just adding zero to find that, for any $\mu\in \calP_\fpsss(B_\rmw)$,
    \[
        \int_H \phi(\bfu) \;\rmd\mu(\bfu) = \int_H \left\{ \phi(\bfu) + \dual{\bfF(\bfu), \bfD\Psi(\bfu)}_{V', V}\right\}\;\rmd\mu(\bfu).
    \]
    Taking the infimum in $\Psi$ yields
    \[
        \int_H \phi(\bfu) \;\rmd\mu(\bfu) = \inf_{\Psi\in\calT^\cyl} \int_H \left\{ \phi(\bfu) + \dual{\bfF(\bfu), \bfD\Psi(\bfu)}_{V', V}\right\}\;\rmd\mu(\bfu).
    \]
    Since this holds for any $\mu\in \calP_\fpsss(B_\rmw)$, we take the maximum in this space to find that
    \begin{multline*}
        \max_{\mu\in\calP_\fpsss(B_\rmw)} \int_H \phi(\bfu) \;\rmd\mu(\bfu) \\
        = \max_{\mu\in\calP_\fpsss(B_\rmw)} \inf_{\Psi\in\calT^\cyl} \int_H \left\{ \phi(\bfu) + \dual{\bfF(\bfu), \bfD\Psi(\bfu)}_{V', V}\right\}\;\rmd\mu(\bfu).
    \end{multline*}

    \textbf{Step 3 (extend to regular measures):} Since $\int_H \phi(\bfu) \;\rmd\mu(\bfu)$ is finite for $\mu\in \calP_\fp(B_\rmw)$, it follows from the selection result in \cref{leminfintinfty} that
    \begin{multline*}
        \inf_{\Psi\in\calT^\cyl} \int_H \left\{ \phi(\bfu) + \dual{\bfF(\bfu), \bfD(\lambda\Psi)(\bfu)}_{V', V} \right\} \;\rmd\mu(\bfu) \\
        = \int_H \phi(\bfu) \;\rmd\mu(\bfu) - \infty = -\infty, \quad \forall \mu \in \calP_\fp(B_\rmw)\setminus \calP_\fpsss(B_\rmw),
    \end{multline*}
    while
    \begin{multline*}
        \inf_{\Psi\in\calT^\cyl} \int_H \left\{ \phi(\bfu) + \dual{\bfF(\bfu), \bfD(\lambda\Psi)(\bfu)}_{V', V} \right\} \;\rmd\mu(\bfu) \\ 
            = \int_H \phi(\bfu) \;\rmd\mu(\bfu) > -\infty, \quad \forall \mu \in \calP_\fpsss(B_\rmw).
    \end{multline*}
    Therefore, when taking the maximum over $\mu$ in $\calP_\fp(B_\rmw)$, it must lie in $\calP_\fpsss(B_\rmw)$. In other words,
    \begin{multline*}
        \max_{\mu\in\calP_\fpsss(B_\rmw)} \inf_{\Psi\in\calT^\cyl} \int_H \left\{ \phi(\bfu) + \dual{\bfF(\bfu), \bfD\Psi(\bfu)}_{V', V}\right\}\;\rmd\mu(\bfu) \\
        = \max_{\mu\in\calP_\fp(B_\rmw)} \inf_{\Psi\in\calT^\cyl} \int_H \left\{ \phi(\bfu) + \dual{\bfF(\bfu), \bfD\Psi(\bfu)}_{V', V}\right\}\;\rmd\mu(\bfu).
    \end{multline*}

    \textbf{Step 4 (Sion's minimax theorem):} Now we apply Sion's Minimax \cref{thmsionminimax}. We have that $\calP_\fp(B_\rmw)$ is a compact topological vector space, and $\calT^\cyl$ is a normed vector space. As for the function on the product space $\calP_\fp(B_\rmw) \times \calT^\cyl$, we consider
    \[ 
        G(\mu, \Psi) = \int_H \left\{ \phi(\bfu) + \dual{\bfF(\bfu), \bfD\Psi(\bfu)}_{V', V}\right\}\;\rmd\mu(\bfu),
    \]
    which is linear in $\mu\in\calP_\fp(B_\rmw)$ and affine in $\Psi\in\calT^\cyl.$ Thanks to \cref{corgvwGcontinuous} and \cref{lemgvwGcontinuousinPsi}, $G=G(\mu,\Psi)$ is continuous in each variable. Thus, \cref{thmsionminimax} applies and we find
    \begin{multline*}
        \max_{\mu\in\calP_\fp(B_\rmw)} \inf_{\Psi\in\calT^\cyl} \int_H \left\{ \phi(\bfu) + \dual{\bfF(\bfu), \bfD\Psi(\bfu)}_{V', V}\right\}\;\rmd\mu(\bfu) \\
        = \inf_{\Psi\in\calT^\cyl} \max_{\mu\in\calP_\fp(B_\rmw)} \int_H \left\{ \phi(\bfu) + \dual{\bfF(\bfu), \bfD\Psi(\bfu)}_{V', V}\right\}\;\rmd\mu(\bfu).
    \end{multline*}

    \textbf{Step 5 (Bauer's maximum principle):} Now we apply Bauer's maximum principle stated in \cref{thmbauer}. We have that $\calP_\fp(B_\rmw)$ is a compact and convex topological vector space, and $\mu \mapsto G(\mu, \Psi)$ is continuous, for every $\Psi\in\calT^\cyl$, thanks to \cref{corgvwGcontinuous}. Thus, the conditions of \cref{thmbauer} are met and we find that
    \begin{multline*} 
        \max_{\mu\in\calP_\fp(B_\rmw)} \int_H \left\{ \phi(\bfu) + \dual{\bfF(\bfu), \bfD\Psi(\bfu)}_{V', V}\right\}\;\rmd\mu(\bfu) \\
          = \max_{\mu \in \Ex(\calP_\fp(B_\rmw))} \int_H \left\{ \phi(\bfu) + \dual{\bfF(\bfu), \bfD\Psi(\bfu)}_{V', V}\right\}\;\rmd\mu(\bfu),
    \end{multline*} 
    where $\Ex(\calP_\fp(B_\rmw))$ is the set of extreme points of $\calP_\fp(B_\rmw).$ Taking the infimum in $\Psi\in\calT^\cyl$ yields
    \begin{multline*}
        \inf_{\Psi\in\calT^\cyl} \max_{\mu\in\calP_\fp(B_\rmw)} \int_H \left\{ \phi(\bfu) + \dual{\bfF(\bfu), \bfD\Psi(\bfu)}_{V', V}\right\}\;\rmd\mu(\bfu) = \\
        \inf_{\Psi\in\calT^\cyl} \max_{\mu \in \Ex(\calP_\fp(B_\rmw))} \int_H \left\{ \phi(\bfu) + \dual{\bfF(\bfu), \bfD\Psi(\bfu)}_{V', V}\right\}\;\rmd\mu(\bfu).
    \end{multline*} 

    \textbf{Step 6 (extremes in $\calP_\fp(B_\rmw)$):} Now we use \cref{propextremePfp} that characterizes the extreme points in $\calP_\fp(B_\rmw)$ to find that
    \begin{multline*}
        \max_{\mu\in\Ex(\calP_\fp(B_\rmw))} \int_H \left\{ \phi(\bfu) + \dual{\bfF(\bfu), \bfD\Psi(\bfu)}_{V', V}\right\}\;\rmd\mu(\bfu) \\
        = \max_{(\bfu_1, \bfu_2, \theta_1, \theta_2)\in \calB_f} \;\sum_{i=1}^2\theta_i\{\phi(\bfu_i) + \dual{\bfF(\bfu_i), \bfD\Psi(\bfu_i)}_{V', V} \},
    \end{multline*}
    where $\calB_f$ is given in \cref{propextremePfp}. Taking the infimum in $\Psi$ yields the last step
    \begin{multline*}
        \inf_{\Psi\in\calT^\cyl}\max_{\mu\in\Ex(\calP_\fp(B_\rmw))} \int_H \left\{ \phi(\bfu) + \dual{\bfF(\bfu), \bfD\Psi(\bfu)}_{V', V}\right\}\;\rmd\mu(\bfu) \\
        = \inf_{\Psi\in\calT^\cyl} \max_{(\bfu_1, \bfu_2, \theta_1, \theta_2)\in \calB_f} \;\sum_{i=1}^2\theta_i\{\phi(\bfu_i) + \dual{\bfF(\bfu_i), \bfD\Psi(\bfu_i)}_{V', V} \}.
    \end{multline*}

    \textbf{Conclusion:} Putting the steps together completes the proof.
\end{proof}

Since any Foias-Prodi stationary statistical solution on $H$ is carried by a bounded ball in $H$, according to \cref{lemmusupportedinH}, the following corollary holds.

\begin{corollary}
    Assume the settings described in \cref{sec3dnse}, in particular with $\Omega\subset\RR^3$ bounded and with $\bff\in V'.$ Then, for any $\phi: V\cap B_H(\radius_0)\rightarrow \RR$ which is weakly-continuous in $V$ and has sub-quadratic growth on $V$ (as in \cref{lemgvwGcontinuous}), it follows that
    \begin{multline}
        \max_{\mu\in\calP_\fpsss(H)} \int_H \phi(\bfu) \;\rmd\mu(\bfu) \\
        = \inf_{\Psi\in\calT^\cyl} \max_{(\bfu_1, \bfu_2, \theta_1, \theta_2)\in \calB_f} \;\sum_{i=1}^2\theta_i\{\phi(\bfu_i) + \dual{\bfF(\bfu_i), \bfD\Psi(\bfu_i)}_{V', V} \},
    \end{multline}
    where $\calB_f$ is given in \cref{propextremePfp}, with $B=B_H(\radius_0),$ where $\radius_0$ is given by \eqref{defrho0}.
\end{corollary}

\begin{proof}
    This follows from \cref{maintheorem}, using that, from \cref{lemmusupportedinH}, $\calP_\fpsss(H) = \calP_\fpsss(B_H(\radius_0))$ and that, according to \eqref{lemPBHr0notempty}, the latter subspace is non-empty.
\end{proof}

\section{Concluding remarks and perspectives}
\label{seconclusions}

Here we present some final considerations and open problems related to our work.

\subsection{The role of stationary statistical solutions}

This work builds on two fundamental lines of research in the theory of differential equations. The first concerns the notion of statistical solution for the Navier-Stokes equations, introduced by Ciprian Foias and Giovanni Prodi and further developed by Mark Vishik, Andrei Fursikov, Roger Temam, and others. The second concerns the minimax formula for optimal bounds on evolutionary systems, advanced by Charlie Doering and collaborators.

The minimax formula has been proved for finite-dimensional ODEs and for the two-dimensional Navier-Stokes equations, and should apply to other well-posed systems, as well, but its extension to the three-dimensional Navier-Stokes equations seemed to be a major challenge, due to the lack of a well-defined semigroup and of a suitable notion of invariant measures.

For circumventing that, the stationary statistical solutions play a fundamental role, as proved here, reaching beyond their usual realm of applications and enabling the extension of the minimax formula to the three-dimensional Navier-Stokes equations, a problem that was central to Doering's program.

This work bridges these two lines of research and, in light of the above remarks, is a landmark for both.

\subsection{Maximizing asymptotic time averages}
\label{secmaxasymptimeavg}

In the works \cite{TobGolDoe2018} and \cite{RosaTemam2022}, the optimal formula is stated not only for ensemble averages but also for asymptotic time averages. More precisely, for the two-dimensional Navier-Stokes equations considered in \cite{RosaTemam2022}, with $\{S(t)\}_{t\geq 0}$ denoting the semigroup associated with the system and with a compact set $K$ which is attracting for the orbits starting from a given positively invariant set $B$, it is shown that
\begin{multline}
    \label{minimaxforasymptotictimeaverages}
    \max_{\bfu_0\in B} \lim_{T\rightarrow \infty} \frac{1}{T}\int_0^T \phi(S(t)\bfu_0)\;\rmd t \\ 
    = \max_{\mu\in\calP_\fpsss(B)} \int_H \phi \;\rmd\mu = \inf_{\Psi\in\calT^\cyl} \max_{\bfu\in K} \;\{\phi + \dual{\bfF, \bfD\Psi}_{V', V} \}.
\end{multline}
The first equality follows from the representation of asymptotic time averages as invariant measures, the fact that maximal invariant measures are ergodic, the Birkhoff-Khinchin Ergodic Theorem, and the fact that the Foias-Prodi stationary statistical solutions are invariant measures for the semigroup, in the two-dimensional case. In the three-dimensional case, however, the lack of a well-defined semigroup in the phase space and, hence, of invariant measures in the classical sense prevent the application of these results. As a consequence, the first equality above is not known to hold in the three-dimensional case. 

An alternative approach is to work with the time-translation semigroup in the space of trajectories, as in \cite{FRT2015,FRT2019}, for which the Vishik-Fursikov stationary statistical solutions are invariant measures. In this case, an analog of the first equality in \eqref{minimaxforasymptotictimeaverages} can be shown to hold. However, the analog of the second equality in trajectory space is not much useful since it involves global trajectories of the Navier-Stokes equations instead of points in phase space. The challenge, in this case, is to relate the minimax formula in the trajectory space to that in phase space.

One thing that should be clear is that, the first equality is based on the Lagrangian invariance of the measure, while the second equality is based on the Eulerian invariance of the measure. The Foias-Prodi stationary statistical solutions are based on the definition of Eulerian invariance, and this is one of the reasons it works well here. The two notions of invariance are not known to be equivalent for the three-dimensional Navier-Stokes equations.

In any case, the full relation \eqref{minimaxforasymptotictimeaverages} is still an open problem for the three-dimensional Navier-Stokes equations and for any other system lacking a well-defined semigroup.

It is worth mentioning that, while \eqref{minimaxformula} does not require $B$ to be positively invariant, the first relation in \eqref{minimaxforasymptotictimeaverages} does.

\subsection{Properties of maximizing measures}

The first equality in \eqref{minimaxforasymptotictimeaverages} is directly connected with the problem of ergodic optimization in dynamical system theory \cite{Jenkinson2019}. Of major interest in ergodic optimization is the characterization of maximizing orbits and maximizing measures. Similarly, in the context of the Navier-Stokes equations, be them two- or three-dimensional, it is also a fundamental and challenging problem to characterize the maximizing stationary statistical solutions, for relevant physical functionals $\phi,$ in general arbitrary domains or in specific geometries. See, for instance, the works \cite{CheGouHuaPa2014,FantuzziWynn2015, FGHC2016,TobGolDoe2018,Goluskin2018,GoluskinFantuzzi2019}, and, on a different line of investigation, \cite{LinThiffeaultDoering2011}, for related results.

\subsection{Maximum with one or two points}
\label{secmaxinoneortwopoints}

The minimax formula \eqref{minimaxformula} in \cref{maintheorem} involves convex combinations of at most two points. We can relate it to minimax formulas involving only single points in two different ways. First, we can simply bound the integrand by its supremum to find that
\begin{multline*}
    \int_H \{\phi + \dual{\bfF, \bfD\Psi}_{V', V}\} \;\rmd\mu \leq \sup_{\bfu\in B} \{\phi(\bfu) + \dual{\bfF(\bfu), \bfD\Psi(\bfu)}_{V', V}\} \int_H \;\rmd\mu \\ = \sup_{\bfu\in B} \{\phi(\bfu) + \dual{\bfF(\bfu), \bfD\Psi(\bfu)}_{V', V}\},
\end{multline*}
for any $\mu \in \calP_\fp(B_\rmw),$ and, hence, deduce that
\begin{multline}
    \label{minimaxupperbound}
    \max_{\mu\in\calP_\fpsss(B)} \int_H \phi \;\rmd\mu = \inf_{\Psi\in\calT^\cyl} \max_{\mu\in\calP_\fp(B_\rmw)} \int_H \{\phi + \dual{\bfF, \bfD\Psi}_{V', V}\} \;\rmd\mu \\
    \leq \inf_{\Psi\in\calT^\cyl} \sup_{\bfu\in B \cap V} \{\phi(\bfu) + \dual{\bfF(\bfu), \bfD\Psi(\bfu)}_{V', V}\}.
\end{multline}
The right hand side formula is the classical version obtained for well-posed systems in \cite{TobGolDoe2018} and \cite{RosaTemam2022}.

On the other hand, using that $\calB_f = \calB_{f, 1} \cup \calB_{f, 2}$ (see \cref{propextremePfp}), we find that
\begin{multline*} 
    \max_{(\bfu_1, \bfu_2, \theta_1, \theta_2)\in \calB_f} \;\sum_{i=1}^2\theta_i\{\phi (\bfu_i) + \dual{\bfF (\bfu_i), \bfD\Psi(\bfu_i)}_{V', V} \} \\
    \geq \max_{(\bfu_1, \bfu_2, \theta_1, \theta_2)\in \calB_{f,1}} \;\sum_{i=1}^2\theta_i\{\phi (\bfu_i) + \dual{\bfF(\bfu_i), \bfD\Psi(\bfu_i)}_{V', V} \} \\
    = \max_{\bfu \in B\cap V, \ell(\bfu) \leq 0} \{\phi(\bfu) + \dual{\bfF(\bfu), \bfD\Psi(\bfu)}_{V', V} \}.
\end{multline*}
This implies that
\begin{multline} 
    \label{minmaxlowerbound}
    \max_{\mu\in\calP_\fpsss(B)} \int_H \phi \;\rmd\mu = \inf_{\Psi\in\calT^\cyl} \max_{(\bfu_1, \bfu_2, \theta_1, \theta_2)\in \calB_f} \;\sum_{i=1}^2\theta_i\{\phi(\bfu_i) + \dual{\bfF(\bfu_i), \bfD\Psi(\bfu_i)}_{V', V} \} \\ 
    \geq \inf_{\Psi\in\calT^\cyl} \max_{\bfu \in B\cap V, \ell(\bfu) \leq 0} \{\phi(\bfu) + \dual{\bfF(\bfu), \bfD\Psi(\bfu)}_{V', V} \}.
\end{multline}
Notice we have squeezed our optimal minimax formula involving convex combinations of at most two points between two similar minimax formulas involving only single points. 

It is an open problem to understand the relation between these three formulas. Clearly, if the two lower and upper single-point bounds \eqref{minmaxlowerbound} and \eqref{minimaxupperbound} agree, then they agree with \eqref{minimaxformula}, and there is no need to optimize for convex combinations of two points. However, it is not known if, or when, this happens or if our optimal bound \eqref{minimaxformula} is strictly above \eqref{minmaxlowerbound} or strictly below \eqref{minimaxupperbound}.

\subsection{Minimax problem as a linear convex optimization problem}
\label{secminimaxaslinearoptimization}

As mentioned in the Introduction, the minimax problem in the right hand side of \eqref{optminimaxintroprevious} can be written as a linear convex optimization problem. More precisely, one can write
\begin{equation}
  \label{minimaxaslinearoptimizationprevious}
    \inf_{\Psi\in\calT^\cyl} \max_{\bfu\in K} \;\{\phi + \dual{\bfF, \bfD\Psi}_{V', V} \} = \inf_{(C,\Psi)\in\RR\times\calT^\cyl, S_{C,\Psi} \geq 0} C,
\end{equation}
where $S_{C,\Psi}(\bfu) = C - \phi(\bfu) - \dual{\bfF(\bfu), \bfD\Psi(\bfu)}_{V', V}$ is a functional on $V\cap B$. The condition $S_{C,\Psi} \geq 0$ means that $S_{C,\Psi}(\bfu) \geq 0,$ for all $\bfu\in K.$ This is an optimization of the linear map $(C,\Psi) \mapsto C$ over a convex set, since $\RR\times\calT^\cyl$ is convex and $S_{C,\Psi}$ is linear in $C$ and $\Psi.$

Similarly, in the three-dimensional case, the minimax problem in the right hand side of \eqref{optminimaxintro} can be written as
\begin{equation}
    \label{minimaxaslinearoptimization}
    \inf_{\Psi} \max_{\bfu_1, \bfu_2, \theta_1, \theta_2} \;\sum_{i=1}^2\theta_i\{\phi(\bfu_i) + \dual{\bfF(\bfu_i), \bfD\Psi(\bfu_i)}_{V', V} \} = \inf_{(C,\Psi)\in\RR\times\calT^\cyl, S_{C,\Psi} \geq 0} C,
\end{equation}
where now $S_{C,\Psi}=S_{C,\Psi}(\bfu_1, \bfu_2, \theta_1, \theta_2)$ is given by
\begin{equation}
    \label{SCPsiinminimaxaslinearoptimization}
    S_{C,\Psi}(\bfu_1, \bfu_2, \theta_1, \theta_2) = C - \sum_{i=1}^2\theta_i\{\phi(\bfu_i) + \dual{\bfF(\bfu_i), \bfD\Psi(\bfu_i)}_{V', V} \},
\end{equation}
and the constraint $S_{C,\Psi} \geq 0$ means that $S_{C,\Psi}(\bfu_1, \bfu_2, \theta_1, \theta_2) \geq 0,$ should hold for all $(\bfu_1, \bfu_2, \theta_1, \theta_2)\in\calB_f.$ This is still a linear optimization problem. 

\subsection{Sum of squares}

The optimization problems \eqref{minimaxaslinearoptimizationprevious} and \eqref{minimaxaslinearoptimization} are not yet suitable for practical computation. The condition $S_{C,\Psi} \geq 0$ is not easy to check in this generality. In fact, it is an undecidable problem \cite{Richardson1969}. If the system $\bfu_t = \bfF(\bfu)$ is finite-dimensional and the function $\bfF=\bfF(\bfu)$ is a polynomial in $\bfu$, then condition $S_{C,\Psi} \geq 0$ is decidable but NP-hard \cite{MurtyKabadi1987}. For practical purposes, one checks instead whether $S_{C,\Psi}$ is a sum of squares (SoS) of polynomials, which we write as $S_{C,\Psi} \in \texttt{SoS}$. More precisely, one restricts the test functionals to polynomials $\calP_m$ of a certain degree $m$ and considers the following approximate optimization problem
\begin{equation}
  \label{minimaxaslinearoptimizationpreviousSoS}
    \inf_{\Psi\in\calT^\cyl} \max_{\bfu\in K} \;\{\phi + \dual{\bfF, \bfD\Psi}_{V', V} \} \leq \inf_{(C,\Psi)\in\RR\times\calP_m, S_{C,\Psi} \in \texttt{SoS}} C
\end{equation}
as a sum of squares relaxation of the original problem \eqref{minimaxaslinearoptimizationprevious} (and similarly for \eqref{minimaxaslinearoptimization}). This problem can be further rewritten as a semidefinite programming problem, which can be solved efficiently in polynomial time. See \cite{Parrilo2003,Lasserre2015} for more details on this topic. 

When the problem is infinite-dimensional and/or the right hand side $\bfF=\bfF(\bfu)$ of the equation is not a polynomial, one must first approximate the problem to be so. For example, for the three-dimensional Navier-Stokes equations, $\bfF=\bfF(\bfu)$ is a quadratic polynomial but the phase space is infinite-dimensional. On a bounded spatial domain, for instance, one may first approximate the equations by a finite-dimensional system, such as a Galerkin approximation $\bfF_k(\bfu) = P_k \bfF(P_k\bfu),$ where $P_k$ is the Galerkin projector associated with the first $k$ eigenmodes of the Stokes operator, and then check the SoS condition. This leads to the problem
\begin{multline}
    \label{minimaxaslinearoptimizationSoS}
    \inf_{\Psi} \max_{\bfu_1, \bfu_2, \theta_1, \theta_2} \;\sum_{i=1}^2\theta_i\{\phi(\bfu_i) + \dual{\bfF_k(\bfu_i), \bfD\Psi(\bfu_i)}_{V', V} \} \\
    \leq \inf_{(C,\Psi)\in\RR\times\calP_m, S_{k,C,\Psi} \in \texttt{SoS}} C,
\end{multline}
where now $S_{k,C,\Psi}=S_{k,C,\Psi}(\bfu_1, \bfu_2, \theta_1, \theta_2)$ is given by
\begin{equation}
    \label{SCPsiinminimaxaslinearoptimizationSoS}
    S_{k,C,\Psi}(\bfu_1, \bfu_2, \theta_1, \theta_2) = C - \sum_{i=1}^2\theta_i\{\phi(\bfu_i) + \dual{\bfF_k(\bfu_i), \bfD\Psi(\bfu_i)}_{V', V} \},
\end{equation}
and the test functionals are restricted to polynomials of degree $m$ in $P_k H.$

When fixing the degree $m$ of the polynomial test functions, on a $k$ dimensional space, the optimization problem becomes a semidefinite program where the decision variable is a symmetric matrix of size $N \times N,$ where $N = \binom{k + m/2}{m/2}$, with the computational complexity of the order of $O(N^3)$ to $O(N^{9/5})$, depending on the algorithm. This is maneageble if the dimension is not too large. Improving this is an active area of research.

While the linear optimization problem \eqref{minimaxaslinearoptimizationpreviousSoS} has been implemented and investigated before in practical problems (see e.g. \cite{CheGouHuaPa2014,FantuzziWynn2015, FGHC2016,TobGolDoe2018,Goluskin2018,GoluskinFantuzzi2019}), the linear optimization problem \eqref{minimaxaslinearoptimizationSoS} is yet to be implemented. As mentioned in \cref{secmaxinoneortwopoints}, the bound in \eqref{optminimaxintroprevious} is an upper bound on the optimal bound in \eqref{optminimaxintro}, and it is not clear whether the bound obtained with \eqref{SCPsiinminimaxaslinearoptimizationSoS} would be a practical improvement. Nevertheless, the optimal formula in \eqref{optminimaxintro} is a fundamental theoretical result that allows us to interpret these approximate formulas as sharp upper bounds.

\subsection{SoS relaxation gap} 

Relaxing the constraint $S_{C,\Psi} \geq 0$ to $S_{C,\Psi} \in \texttt{SoS}$ potentially leads to a gap in the optimal value of the optimization problem, which can be termed as the SoS relaxation gap. This gap is a fascinating problem dating back to Hilbert and connected with Hilbert's 17th problem \cite{Hilbert1900}. Hilbert \cite{Hilbert1888} proved that not all polynomials in several variables can be written as sums of squares of polynomials, except in some special cases (polynomials of a single variable, quadratic forms, and polynomials of degree four in two variables). This lead to Hilbert's 17th problem of whether all positive polynomials can be written instead as sums of squares of rational functions, which was proved to be true by Artin \cite{Artin1927}. Checking whether polynomials are sums of squares of rational functions can still be solved in polynomial time once the dimension of the denominator is also fixed, but the computational cost escalates much faster, and usually only the case of SoS of polynomials or SoS of rational functions with low-dimensional denominators are considered, for practical reasons. In any case, the problem is expected to have a theoretical SoS relaxation gap, which is still not well understood in the context of optimal bounds for dynamical systems.

\subsection{Optimal upper bound formula for enstrophy and energy dissipation}
\label{secenstrophyupperbound}

For \cref{maintheorem}, it is required that $\phi: V\cap B\rightarrow \RR$ be weakly-continuous in $V$ and with sub-quadratic growth on $V$ (as in \cref{lemgvwGcontinuous}). Thus we cannot apply the result to 
\[ \phi(\bfu) = \|\bfnabla \bfu\|^2,
\]
which prevents us from obtaining optimal upper bound formulas for the mean enstrophy and the mean kinetic energy dissipation rate, defined by
\[ \frac{\rho}{2}\int_H \|\bfnabla\times\bfu\|_{L^2}^2 \;\rmd\mu(\bfu) \quad \textrm{and} \quad \rho\nu\int_H \|\bfnabla\bfu\|_{L^2}^2 \;\rmd\mu(\bfu),
\]
respectively. In the divergence-free case with no-slip boundary conditions, it follows that $\|\bfnabla\times\bfu\|_{L^2}^2 = \|\bfnabla\bfu\|_{L^2}^2,$ so both expressions are multiples of $\|\bfnabla\bfu\|_{L^2}^2$, which is neither weakly continuous, nor with sub-quadratic growth, in $V$. As mentioned in the Introduction, it is possible to extend \cref{maintheorem} to hold for weakly upper-semicontinuous functions on $V$, allowing us to apply the result to $-\|\bfnabla \bfu\|^2$ and obtain an optimal lower bound formula for the desired physical quantities. But the corresponding upper bound formula does not follow from our main result.

Nevertheless, a formula can still be obtained by approximating the function via Galerkin approximations, namely
\begin{multline}
    \label{minimaxformulawithPkenstrophy}
    \sup_{\mu\in\calP_\fpsss(B)} \int_H \|\bfnabla \bfu\|_{L^2}^2 \;\rmd\mu(\bfu) \\ = \lim_{k\rightarrow \infty} \inf_{\Psi\in\calT^\cyl} \max_{(\bfu_1, \bfu_2, \theta_1, \theta_2)\in \calB_f} \;\sum_{i=1}^2\theta_i\{\|\bfnabla P_k\bfu_i\|_{L^2}^2 + \dual{\bfF(\bfu_i), \bfD\Psi(\bfu_i)}_{V', V} \}.
\end{multline}
Notice that the supremum appears on the left hand side above, in contrast to the formula in \cref{maintheorem}, because the function $\mu \mapsto \int_H \|\bfnabla \bfu\|_{L^2}^2 \;\rmd\mu(\bfu)$ is not continuous on $\calP_\fpsss(B).$

In order to establish \eqref{minimaxformulawithPkenstrophy}, we consider the functional
\[
    \phi_k(\bfu) = \|\bfnabla P_k\bfu\|_{L^2}^2,
\]
which is defined on all $H$. Since the Galerkin projectors $P_k$ are orthogonal not only in $H$ but also in $V$, they satisfy $\phi_k(\bfu) \leq \phi(\bfu),$ for each $\bfu.$ This means that
\[ \int_H \phi_k(\bfu) \;\rmd\mu(\bfu) \leq \int_H \phi(\bfu) \;\rmd\mu(\bfu),
\]
so that,
\[ \max_{\mu\in\calP_\fpsss(B)} \int_H \phi_k(\bfu) \;\rmd\mu(\bfu) \leq \sup_{\mu\in\calP_\fpsss(B)} \int_H \phi(\bfu) \;\rmd\mu(\bfu),
\]
for each $k\in\NN.$  Taking the limit in $k$, we find that 
\begin{equation} 
    \label{limsupenstro}
    \limsup_{k\rightarrow\infty}\max_{\mu\in\calP_\fpsss(B)} \int_H \phi_k(\bfu) \;\rmd\mu(\bfu) \leq \sup_{\mu\in\calP_\fpsss(B)} \int_H \phi(\bfu) \;\rmd\mu(\bfu).
\end{equation}
On the other hand, for an arbitrary $\mu_0$ in $\calP_\fpsss(B),$ since $\phi_k(\bfu)$ converges monotonically to $\phi(\bfu)$ from below, pointwise in $\bfu\in H,$ we obtain, from the Monotone Convergence Theorem, that
\[
    \int_H \phi(\bfu) \;\rmd\mu_0(\bfu) = \lim_{k\rightarrow\infty} \int_H \phi_k(\bfu) \;\rmd\mu_0(\bfu).
\]
Since 
\[  \int_H \phi_k(\bfu) \;\rmd\mu_0(\bfu) \leq \max_{\mu\in\calP_\fpsss(B)} \int_H \phi_k(\bfu) \;\rmd\mu(\bfu),
\] 
for each $k\in\NN,$ we obtain
\[
    \int_H \phi(\bfu) \;\rmd\mu_0(\bfu) \leq \liminf_{k\rightarrow\infty} \max_{\mu\in\calP_\fpsss(B)} \int_H \phi_k(\bfu) \;\rmd\mu(\bfu).
\]
Taking the supremum in $\mu_0\in \calP_\fpsss(B),$ we find that 
\begin{equation}
    \label{liminfenstro}
    \sup_{\mu_0\in\calP_\fpsss(B)}\int_H \phi(\bfu) \;\rmd\mu_0(\bfu) \leq \liminf_{k\rightarrow\infty} \max_{\mu\in\calP_\fpsss(B)} \int_H \phi_k(\bfu) \;\rmd\mu(\bfu).
\end{equation}

Combining inequalities \eqref{limsupenstro} and \eqref{liminfenstro} yields
\[
    \sup_{\mu\in\calP_\fpsss(B)} \int_H \phi(\bfu) \;\rmd\mu(\bfu) = \lim_{k\rightarrow \infty} \max_{\mu\in\calP_\fpsss(B)} \int_H \phi_k(\bfu) \;\rmd\mu(\bfu).
\]
Then, applying \cref{maintheorem} to each $\phi_k$, we obtain \eqref{minimaxformulawithPkenstrophy}.

We should mention that, in practice, one usually uses anyway the Galerkin approximation to reduce the problem to a finite-dimensional optimization problem, so formula \eqref{minimaxformulawithPkenstrophy} actually justifies this approach.

For the result in \eqref{minimaxformulawithPkenstrophy}, we only use that $\phi_k(\bfu) = \phi(P_k\bfu)$ increases monotonically towards $\phi(\bfu)$ from below. This same property holds for the functional $\phi(\bfu) = \|(I-P_m)\bfnabla\bfu\|_{L^2}^2$, for a fixed $m\in\NN,$ so the same formula applies in this case. This should be useful when estimating the enstrophy or the energy dissipation rate associated with the smallest length scales of the flow.

The main problem in switching the limit in $k$ with the inf-sup problem in the right hand side of \eqref{minimaxformulawithPkenstrophy} is the lack of continuity of the functional $\bfu \mapsto \|\bfnabla \bfu\|_{L^2}^2$ over the compact set $\calB_f$ (or the lack of compactness of $\calB_f,$ if we were to endow it with the topology inherithed from $V$).\footnote{A simple example where switching the limit in $k$ with the inf-sup leads to a different result is one in which $\lim_k \inf_{\Psi\in \calT}\sup_{u\in B} (\phi_k(u) + L(\Psi, u)) = -1 < 0 = \inf_{\Psi\in \calT}\sup_{u\in B} (\phi(u) + L(\Psi, u))$, where $B = [0, 1]$ (compact), $\phi(u) = 0$ for $0 < u \leq 1$ and $\phi(0) = -1$ (lower semi-continuous), $\calT = \{\Psi(u) = au; \;a\in \RR\}$ (normed vector space), $\phi_k(u) = \min\{ku - 1, 0\}$ (continuous and increasing monotonically towards $\phi(u)$), $F(u) = -u$ (continuous), and $L(\Psi, u) = \Psi'(u)F(u) = -au$ (linear and continuous). This gap is independent of the topology we use in $B$, so if we change it to the discrete topology, we still have this gap, now with $\phi(u)$ continuous but with $B$ non-compact.} If $\mu$ is carried on a more regular set, then a similar result goes through, as discussed in \cref{secvfsss}. 

\subsection{Vishik-Fursikov stationary statistical solutions}
\label{secvfsss}

We have considered here the problem of maximizing ensemble averages over the set of arbitrary Foias-Prodi stationary statistical solutions. This type of solution has finite mean enstrophy, which allows us to apply our results to certain types of functions which are weakly-continuous in $V$ and have sub-quadratic growth on $V$ (as in \cref{lemgvwGcontinuous}), such as the kinetic energy and the energy-transfer terms. For the mean enstrophy itself, one needs more regularity. This can be achieved by considering domains $\Omega$ with smooth boundaries and with forcing terms $\bff$ in $H$, and then by restricting the maximum in the left hand side of the minimax formula \eqref{minimaxformula} to more regular types of stationary statistical solutions, such as Vishik-Fursikov stationary statistical solutions. These satisfy the bound (see \cite[Theorem 5.4]{FRT2019})
\[
    \int_H \|A\bfu\|_{L^2}^{2/3} \;\rmd\mu(\bfu) \leq \lambda_1 \nu^2 G^2,
\]
where $\lambda_1$ is the first eigenvalue of the Stokes operator $A$ and $G = \|\bff\|_{L^2} / \nu^2 \lambda_1^{3/4}$ is the Grashof number. This type of solution includes the time-average stationary statistical solutions discussed in \cref{secsss}.

This bound, however, does not follow directly from the Liouville-type transport equation \eqref{eqliouvilleintro} and must be imposed in the space as an additional constraint, which leads, in turn, to a different minimax formula, involving a convex combination of at most three points. This will be addressed elsewhere.

\subsection{Other systems}

The optimal minimax formula in \cref{maintheorem} is stated for the three-dimensional Navier-Stokes equations, but it should be possible to adapt it to a much more general class of systems. Typically, one expects the formula \eqref{optminimaxintroprevious} to hold for systems which do not require any extra energy-type constraint, while the formula \eqref{optminimaxintro} is expected to hold for systems which require an energy-type constraint. Systems with two or more constraints would be expected to have a minimax formula involving convex combinations of three or more points. In this regard, it would be interesting to investigate similar problems, especially those with no global well-posedness results and with extra energy-type or entropy-type constraints, such as 3D liquid crystal flow equations (Ericksen-Leslie system), 3D nonlinear Schr\"odinger equation, 3D aggregation-diffusion equations (Keller-Segel), 1D aggregation-diffusion equations with degenerate kernels, systems of hyperbolic conservation laws, 1D Camassa-Holm equation, 1D Hunter-Saxton Equation, and others. The one-dimensional equations are specially suitable for numerical investigations since they are more tractable computationally.

\section*{Acknowledgments}

This work is dedicated to the memory of Ciprian Foias (1933--2020) and Charlie Doering (1953--2021), and is also a tribute to Roger Temam on the occasion of his 85th birthday (May 2025). We would like not only to acknowledge their fundamental contributions on the two major lines of research connected by this work, but also to thank them for their inspiration, mentorship, and friendship over the years.


Authors AB and RMSR acknowledge support from Coordenação de Aperfeiçoamento de Pessoal de Nível Superior (CAPES), Brasil, grant no. 001, and Conselho Nacional de Desenvolvimento Cient\'{\i}fico e Tecnol\'ogico (CNPq), Bras\'{\i}lia, Brasil, grant no. 408751/2023-1. Author CFM received support from the National Science Foundation under the grant DMS-2239325.

\section*{Data Availability}

No datasets were generated or analyzed for this work, and therefore data sharing is not applicable.

\end{document}